\documentclass[a4paper,11pt]{amsart}
\usepackage[left=2.7cm,right=2.7cm,top=3.5cm,bottom=3cm]{geometry}

\usepackage{amssymb,latexsym,amsmath,amsthm,amscd}
\usepackage{stmaryrd}

\usepackage{graphicx}
\usepackage{dsfont}
\usepackage{longtable,tabu}
\usepackage[all]{xy}
\usepackage{mathrsfs}
\usepackage{color}
\definecolor{Blue}{rgb}{0.3,0.3,0.9}
\usepackage{array}
\usepackage{tabu}
\usepackage{hyperref}

\usepackage[T2A,OT1]{fontenc}
\DeclareSymbolFont{cyrillic}{T2A}{cmr}{m}{n}
\DeclareMathSymbol{\Sha}{\mathalpha}{cyrillic}{216}

\newcommand{\sk}{\vspace{0.1in}}
\newtheorem{thm}{Theorem}[section]
\newtheorem{pro-thm}[thm]{Main Theorem}
\newtheorem{def-thm}[thm]{Definition-Theorem}
\newtheorem{cor}[thm]{Corollary}
\newtheorem{lem}[thm]{Lemma}
\newtheorem{lm}[thm]{Lemma}
\newtheorem{prop}[thm]{Proposition}

\newtheorem{conj}[thm]{Conjecture}

\newtheorem*{mainthmA}{Theorem~A}
\newtheorem*{mainthmB}{Corollary~B}
\newtheorem*{mainthmC}{Corollary~C}
\theoremstyle{definition}

\theoremstyle{remark}
\newtheorem{rem}[thm]{Remark}
\newtheorem{intro-rem}{Remark}
\numberwithin{equation}{section}
\newtheorem{question}[thm]{Question}

\newcommand{\abs}[1]{\left\vert#1\right\vert}

\newcommand{\eps}{\varepsilon}

\newcommand{\pp}{\mathfrak{p}}

\newcommand{\bQ}{\mathbf{Q}}
\newcommand{\bZ}{\mathbf{Z}}
\newcommand{\bC}{\mathbf{C}}

\def\cF{{\mathcal F}}  

\def\cH{{\mathcal H}}

\def\cK{{\mathcal K}}  

\def\cO{\mathcal O}

\def\cW{{\mathcal W}}

\def\cX{\mathcal X}

\def\cQ{\mathcal Q}

\newcommand{\Q}{\mathbf{Q}}
\newcommand{\Z}{\mathbf{Z}}

\def\Qp{\Q_p}
\def\Qbar{\ol{\Q}}
\def\Zbar{\ol{\Z}}

\def\Zp{\Z_p}

\def\makeop#1{\expandafter\def\csname#1\endcsname
	{\mathop{\rm #1}\nolimits}\ignorespaces}
\makeop{Hom}   \makeop{End}   \makeop{Aut}   
\makeop{Pic} \makeop{Gal}       \makeop{Div} \makeop{Lie}
\makeop{PGL}   \makeop{Corr} \makeop{PSL} \makeop{sgn} \makeop{Spf}
\makeop{Tr} \makeop{Nr} \makeop{Fr} \makeop{disc}
\makeop{Proj} \makeop{supp} \makeop{ker}   \makeop{Im} \makeop{dom}
\makeop{coker} \makeop{Stab} \makeop{SO} \makeop{SL} \makeop{SL}
\makeop{Cl}    \makeop{cond} \makeop{Br} \makeop{inv} \makeop{rank}
\makeop{id}    \makeop{Fil} \makeop{Frac}  \makeop{GL} \makeop{SU}
\makeop{Trd}   \makeop{Sp} \makeop{Tr}    \makeop{Trd} \makeop{Res}
\makeop{ind} \makeop{depth} \makeop{Tr} \makeop{st} \makeop{Ad}
\makeop{Int} \makeop{tr}    \makeop{Sym} \makeop{can} \makeop{SO}
\makeop{torsion} \makeop{GSp} \makeop{Tor}\makeop{Ker} \makeop{rec}
\makeop{Ind} \makeop{Coker}
\makeop{vol} \makeop{Ext} \makeop{gr} \makeop{ad}
\makeop{Gr}\makeop{corank} \makeop{Ann}
\makeop{Hol} 
\makeop{Fitt} \makeop{Mp} \makeop{CAP}

\def\Sel{\mathrm{Sel}}

\def\Ord{{\mathrm{ord}}}

\def\ndivides{\nmid}

\def\x{{\times}}

\def\al{\alpha}

\def\Lam{\Lambda}

\def\om{\omega}
\def\dirlim{\varinjlim}
\def\prolim{\varprojlim}
\def\iso{\simeq}
\def\con{\equiv}

\newcommand\stt[1]{\left\{#1\right\}}
\newcommand{\beqcd}[1]{\begin{equation*}\label{#1}\tag{#1}}
\newcommand{\eeqcd}{\end{equation*}}
\def\ep{\epsilon}

\def\lam{\lambda}

\def\sg{\sigma}

\def\parref#1{\ref{#1}}
\def\thmref#1{Theorem~\parref{#1}}
\def\expectthmref#1{Expected theorem~\parref{#1}}
\def\propref#1{Proposition~\parref{#1}}
\def\corref#1{Corollary~\parref{#1}}     \def\remref#1{Remark~\parref{#1}}
\def\secref#1{\S\parref{#1}}
\def\lmref#1{Lemma~\parref{#1}}

\newcommand{\dBr}[1]{\llbracket{#1}\rrbracket}

\newcommand{\cR}{\mathbb{I}}

\newcommand{\pwseries}[1]{[[#1]]}

\def\k{\kappa}

\newcommand{\bV}{\mathbb{V}}
\newcommand{\da}{\dagger}
\newcommand{\bff}{{\boldsymbol{f}}}
\newcommand{\bfg}{{\boldsymbol{g}}}
\newcommand{\bfh}{{\boldsymbol{h}}}

\def\bfkappa{{\boldsymbol \kappa}}

\newcommand{\pair}[2]{\langle #1, #2\rangle}
\newcommand{\pairing}{\pair{\,}{\,}}

\begin{document}
	
\title
{On the non-vanishing of generalized Kato classes for elliptic curves of rank $2$}
\author[F.~Castella]{Francesc Castella}
\author[M.-L.~Hsieh]{Ming-Lun Hsieh}
	
\thanks{The first author was partially supported by NSF grant DMS-1801385.}
\thanks{The second author was partially supported by a MOST grant 108-2628-M-001-009-MY4 .}
\subjclass[2010]{Primary 11G05; Secondary 11G40}
\keywords{Elliptic curves, Birch and Swinnerton-Dyer conjecture, $p$-adic families of modular forms, $p$-adic $L$-functions, Euler systems}	
\date{\today}
		
\address[]{Department of Mathematics, University of California Santa Barbara, CA 93106, USA}
\email{castella@ucsb.edu}
\address[]{Institute of Mathematics, Academia Sinica, Taipei 10617, Taiwan}
\email{mlhsieh@math.sinica.edu.tw}




\begin{abstract}
We prove the first cases of a conjecture by Darmon--Rotger on the non-vanishing of generalized Kato classes attached to elliptic curves $E$ over $\bQ$ of rank $2$. Our method also shows that the non-vanishing of generalized Kato classes implies that the
$p$-adic Selmer group of $E$ is $2$-dimensional. The main novelty in the proof is a formula 
for the leading term at the trivial character of an anticyclotomic $p$-adic $L$-function attached to $E$ 
in terms of the derived $p$-adic height of generalized Kato classes and an enhanced $p$-adic regulator.

\end{abstract}


\maketitle
\setcounter{tocdepth}{1}
\tableofcontents

\section{Introduction}
\def\rmH{{\rm H}}

\subsection{Motivating question}\label{sec:motiv}

Let $E$ be an elliptic curve over $\bQ$ (hence modular by \cite{BCDT}), and let $L(E,s)$ be its Hasse--Weil $L$-series. A major advance towards the Birch and Swinnerton-Dyer conjecture was the proof by Gross--Zagier \cite{GZ} and Kolyvagin \cite{kol88} of the implication 
\begin{equation}\label{eq:BSD01}
{\rm ord}_{s=1}L(E,s)=1\quad\Longrightarrow\quad {\rm ord}_{s=1}L(E,s)={\rm rank}_\bZ E(\bQ)
\end{equation}
The proof of $(\ref{eq:BSD01})$ resorts to choosing an imaginary quadratic field $K$ for which the construction of Heegner points on $E$ (over ring class extensions of $K$) becomes available and such that ${\rm ord}_{s=1}L(E/K,s)=1$. By the Gross--Zagier formula, 
the basic Heegner point $y_K\in E(K)$ is then non-torsion, which by Kolyvagin's work implies that $E(K)$ has rank $1$. 
Since $y_K$ descends to $E(\bQ)$ precisely when $L(E,s)$ vanishes to odd order at $s=1$, the above implication follows.

A more recent major advance is Skinner's converse \cite{skinner} to the theorem of Gross--Zagier and Kolyvagin, taking the form of the implication 
\begin{equation}\label{eq:skinner}
\textrm{${\rm rank}_{\bZ}E(\bQ)=1$ and $\#\Sha(E/\bQ)[p^\infty]<\infty$}\quad\Longrightarrow\quad{\rm ord}_{s=1}L(E,s)=1
\end{equation}
for certain primes $p$ of good ordinary reduction for $E$. 
The proof of $(\ref{eq:skinner})$ uses progress \cite{wanIMC} towards an Iwasawa main conjecture over an auxiliary imaginary quadratic field $K$ as before, which under the hypotheses of $(\ref{eq:skinner})$ implies  $y_K\notin E(\bQ)_{\rm tors}$ by the $p$-adic Gross--Zagier formula of \cite{bdp1}, yielding the conclusion by the classical Gross--Zagier formula.

It is natural to wonder about the extension of these results for elliptic curves $E/\bQ$ of rank $2$. Since a stumbling block in this setting is the lack of a systematic construction of algebraic points on $E$ playing the role of Heegner points, a most urgent question to ask might be the following:

\begin{question}\label{motive}
Let $E$ be an elliptic curve over $\bQ$ of rank $2$, and choose an imaginary quadratic field $K$ with 
\begin{equation}\label{eq:ran-K}
{\rm ord}_{s=1}L(E/K,s)={\rm ord}_{s=1}L(E,s)=2. 
\end{equation}
Can one use $K$ to produce explicit \emph{nonzero} classes in the $p$-adic Selmer group ${\rm Sel}(\bQ,V_pE)$ for suitable primes $p$?	
\end{question}

Here ${\rm Sel}(\Q,V_pE)$ denotes the inverse limit under the multiplication-by-$p$ maps of the usual $p^n$-descent Selmer groups ${\rm Sel}_{p^n}(E/\bQ)\subset{\rm H}^1(\bQ,E[p^n])$ tensored with $\bQ_p$, thus sitting in the exact sequence
\[
0\rightarrow E(\bQ)\otimes_{\bZ}\bQ_p\rightarrow {\rm Sel}(\Q,V_pE)\rightarrow T_p\Sha(E/\Q)\otimes_{\bZ_p}\bQ_p\rightarrow 0,
\]
where $T_p\Sha(E/\bQ)$ 
should be trivial, since $\Sha(E/\bQ)$ is expected to be finite. 

In this paper, for good ordinary primes $p$, we provide an affirmative answer to Question~\ref{motive}, with condition $(\ref{eq:ran-K})$ replaced by an algebraic counterpart: 
\begin{equation}\label{eq:ralg-K}
{\rm rank}_{\bZ}E(K)={\rm rank}_{\bZ}E(\bQ)=2\quad\textrm{and}\quad\textrm{$\#\Sha(E/\bQ)[p^\infty]<\infty$}.\nonumber
\end{equation}
Moreover, we prove analogues of the implications
\[
y_K\notin E(\bQ)_{\rm tors}\quad\Longrightarrow\quad{\rm dim}_{\bQ_p}{\rm Sel}(\bQ,V_pE)=1
\]
and 
\[
{\rm rank}_{\bZ}E(\bQ)=1\quad\textrm{and}\quad\textrm{$\#\Sha(E/\bQ)[p^\infty]<\infty$}\quad\Longrightarrow\quad y_K\notin E(\bQ)_{\rm tors}
\]
appearing in the course of $(\ref{eq:BSD01})$ and $(\ref{eq:skinner})$, respectively, in the rank $2$ setting, with $y_K$ replaced by certain generalized Kato classes in ${\rm Sel}(\bQ,V_pE)$. 

\subsection{A conjecture of Darmon--Rotger for rank $2$ elliptic curves}
\label{subsec:DR-conj}

Following their spectacular  
work \cite{DR2} on the Birch and Swinnerton-Dyer conjecture for elliptic curves twisted by certain degree four Artin representations, Darmon--Rotger formulated in \cite{DR2.5} 
a non-vanishing criterion for the generalized Kato classes introduced in \cite{DR2}. In this paper, we consider the special case of their conjectures concerned with elliptic curves of rank $2$.


Let $E/\bQ$ be an elliptic curve of conductor $N$, and let $K$ be an imaginary quadratic field of discriminant prime to $N$. Fix a prime $p>2$ of good ordinary reduction for $E$, and assume that $p=\pp\overline{\pp}$ splits in $K$. Let 
$\chi:G_K={\rm Gal}(\overline{\bQ}/K)\rightarrow\bC^\times$ be a ring class character of conductor prime to $Np$ with $\chi(\overline\pp)\neq\pm{1}$, and set 
$\alpha:=\chi(\overline{\pp})$, $\beta:=\chi(\pp)$.

Let $f\in S_2(\Gamma_0(N))$ be the newform associated with $E$ by modularity,  
so that $L(E,s)=L(f,s)$, 
and let $g$ and $h$ be the weight $1$ theta series of $\chi$ and $\chi^{-1}$, respectively. 
As explained in \cite{DR2.5} (in which $g$ and $h$ can be more general weight $1$ eigenforms), attached to the triple $(f,g,h)$ and the prime $p$ one has four 
\emph{generalized Kato classes}
\begin{equation}\label{eq:4kato}
\kappa(f,g_{\alpha},h_{\alpha^{-1}}),\;\kappa(f,g_{\alpha},h_{\beta^{-1}}),\;\kappa(f,g_{\beta},h_{\alpha^{-1}}),
\;\kappa(f,g_{\beta},h_{\beta^{-1}})\;\in\rmH^1(\bQ,V_{fgh}),
\end{equation}
where  
$V_{fgh}\iso V_pE\otimes V_g\otimes V_h$ is the tensor product of the $p$-adic representations associated to $f$, $g$, and $h$. The class $\kappa(f,g_{\alpha},h_{\alpha^{-1}})$ arises as the $p$-adic limit
\begin{equation}\label{eq:limit}
\kappa(f,g_{\alpha},h_{\alpha^{-1}})=\lim_{\ell\to 1}\k(f,\bfg_\ell,\bfh_\ell)\nonumber
\end{equation}
as $(\bfg_\ell,\bfh_\ell)$ runs over the classical weight $\ell\geqslant 2$ specializations of
Hida families $\bfg$ and $\bfh$ passing through the $p$-stabilizations 
\[
g_\alpha:=g(q)-\beta g(q^p),\quad h_{\alpha^{-1}}:=h(q)-\beta^{-1}h(q^p),
\]
in weight $1$,
and where $\k(f,\bfg_\ell,\bfh_\ell)$ is obtained from the $p$-adic \'etale Abel--Jacobi image of certain
higher-dimensional 
Gross--Kudla--Schoen diagonal
cycles \cite{gross-kudla,gross-schoen}
on triple products of modular curves. 

One of the main results of \cite{DR2} is an explicit reciprocity law  (just stated for $\kappa(f,g_{\alpha},h_{\alpha^{-1}}$) here) of the form
\begin{equation}\label{eq:DR-ERL}
{\rm exp}^*({\rm res}_p(\kappa(f,g_\alpha,h_{\alpha^{-1}})))=(\textrm{nonzero constant})\cdot L(f\otimes g\otimes h,1),
\end{equation}
whereby the classes $(\ref{eq:4kato})$ land in the Bloch--Kato Selmer group
${\rm Sel}(\bQ,V_{fgh})\subset\rmH^1(\bQ,V_{fgh})$ precisely when the triple product $L$-series $L(f\otimes g\otimes h,s)$ vanishes at $s=1$; the main conjecture of \cite{DR2.5} went further to predict that these 
classes span a \emph{non-trivial} subspace of ${\rm Sel}(\bQ,V_{fgh})$ precisely when $L(f\otimes g\otimes h,s)$ vanishes to order exactly $2$ at $s=1$. 

Since for our specific $g$ and $h$ we have the factorization
\begin{align}\label{eq:factorL}
L(f\otimes g\otimes h,s)=L(E,s)\cdot L(E^{K},s)\cdot L(E/K,\chi^2,s),
\end{align}
where $E^{K}$ is the $K$-quadratic twist of $E$, arising from the decomposition
\begin{equation}\label{eq:triple}
V_{fgh}\iso\bigl(V_pE\otimes{\rm Ind}_K^\bQ\mathds{1}\bigr)
\oplus\bigl(V_pE\otimes{\rm Ind}_K^\bQ\chi^2\bigr),
\end{equation}
the cases of the main conjecture of Darmon--Rotger concerned with elliptic curves of rank $2$  
may be stated as follows, where we let 
\begin{equation}\label{eq:E-kato}
\kappa_{\alpha,\alpha^{-1}},\;\kappa_{\alpha,\beta^{-1}},\; \kappa_{\beta,\alpha^{-1}},\;\kappa_{\beta,\beta^{-1}}\in \rmH^1(\bQ,V_pE)
\end{equation}
be the natural image of the classes $(\ref{eq:4kato})$ under the projection $\rmH^1(\bQ,V_{fgh})\rightarrow\rmH^1(\bQ,V_pE)$. 

\begin{conj}[Darmon--Rotger]\label{conj:DR-conj}
Assume that $L(E^{K},1)$ and $L(E/K,\chi^2,1)$ are both non-zero. Then the following are equivalent:
	\begin{itemize}
	\item[(1)] The classes $(\ref{eq:E-kato})$ span a non-trivial subspace of ${\rm Sel}(\Q,V_pE)$.
     \item[(2)] $\dim_{\bQ_p}{\rm Sel}(\bQ,V_pE)=2$.
	\item[(3)] ${\rm rank}_{\bZ}E(\bQ)=2$. 
	\item[(4)] ${\rm ord}_{s=1}L(E,s)=2$.		
	\end{itemize}
\end{conj}

\begin{rem}
Of course, the equivalence of $(2)\Leftrightarrow(3)$ amounts to the finiteness of $\Sha(E/\bQ)[p^\infty]$, and the equivalence $(3)\Leftrightarrow(4)$ is the rank $2$ case of the Birch--Swinnerton-Dyer conjecture. 
\end{rem}


Conjecture~\ref{conj:DR-conj} is a special case of \cite[Conj.~3.2]{DR2.5} and testing the predicted non-vanishing criterion for $(\ref{eq:E-kato})$ experimentally presented an ``interesting challenge'' at the time of its formulation (see [\emph{loc.cit.}, $\S{4.5.3}$]). As an application of the main results of this paper,  
numerical examples supporting this conjecture will be presented in \secref{S:example}. 


\subsection{Main results}

Let $\bar{\rho}_{E,p}:G_\Q={\rm Gal}(\overline{\bQ}/\bQ)\rightarrow{\rm Aut}_{\mathbf{F}_p}(E[p])$ be the residual Galois representation associated to $E$, and write
\[
N=N^+N^-
\] 
with $N^+$ (resp. $N^-$) divisible only by primes which are split (resp. inert) in $K$. Consider the \emph{strict} Selmer group defined by 
\[
\Sel_{\rm str}(\bQ,V_pE):={\rm ker}\biggl(\Sel(\bQ,V_pE)\stackrel{{\rm log}_p}
\rightarrow\Qp\biggr),
\] 
where $\log_p$
denotes the composition of the restriction map ${\rm loc}_p:{\rm Sel}(\bQ,V_pE)\rightarrow E(\bQ_p)\widehat{\otimes}\bQ_p$ with the formal group logarithm $E(\bQ_p)\widehat\otimes\bQ_p\rightarrow\bQ_p$.

\begin{mainthmA}\label{thm:main}
Assume that $L(E^{K},1)$ and $L(E/K,\chi^2,1)$ are both nonzero, and that 
\begin{itemize}
	\item $\bar{\rho}_{E,p}$ is irreducible,
	\item $N^-$ is square-free,
    \item $\bar{\rho}_{E,p}$ is ramified at every prime $q\vert N^-$.
\end{itemize}  
Then $\kappa_{\al,\beta^{-1}}=\kappa_{\beta,\al^{-1}}=0$ and the following statements are equivalent:
\begin{itemize}
	\item[(i)] The class $\kappa_{\al,\al^{-1}}$ is a non-trivial element in $\Sel(\bQ,V_pE)$.
	\item[(ii)] ${\rm dim}_{\bQ_p}{\rm Sel}_{\rm str}(\bQ,V_pE)=1$.
\end{itemize}
\end{mainthmA}	

 %

\begin{rem}
The hypotheses in Theorem~A imply in particular that $E$ has root number $+1$, 
and either of the statements (i) or (ii) implies that $L(E,1)=0$ 
by 
\cite{Kato295}. Thus the elliptic curves in Theorem~A all satisfy 
\begin{equation}\label{eq:ran2}
{\rm ord}_{s=1}L(E,s)\geqslant 2.
\end{equation}
On the other hand, if the root number of $E$ is $+1$ and $\bar{\rho}_{E,p}$ is irreducible and ramified at some prime $q$, by \cite{bfh90} and \cite{vatsal-special} there exist infinitely many imaginary quadratic fields $K$ and ring class characters $\chi$ of prime-power conductor such that the following hold:
\begin{itemize}
	\item $q$ is inert in $K$,
	\item every prime factor of $N/q$ splits in $K$,
	\item $L(E^K,1)\neq 0$ and $L(E/K,\chi^2,1)\neq 0$.
\end{itemize}
Therefore, by Theorem~A the generalized Kato classes $(\ref{eq:E-kato})$ provide an explicit construction of non-trivial Selmer classes for rank $2$  elliptic curves 
analogous to the construction of Heegner classes for rank $1$ elliptic curves. The tables of \secref{S:example} exhibit numerical examples satisfying the rank part of the BSD conjecture and the hypotheses of Theorem~A, yielding the first instances of non-trivial generalized Kato classes for rank $2$ elliptic curves. 
\end{rem}

\begin{rem}
Another construction  
of non-trivial classes in ${\rm Sel}(\bQ,V_pE)$ for elliptic curves $E/\bQ$ satisfying $(\ref{eq:ran2})$ will appear in forthcoming work by Skinner--Urban (see \cite{SU-ICM,urban-tifr}).    
Their construction of Selmer classes 
is completely different from that of generalized Kato classes, and it would be very interesting to compare the two constructions.
\end{rem}
 We obtain immediately from Theorem~A the following result towards Conjecture~\ref{conj:DR-conj}.\begin{mainthmB}
Let the hypotheses be as in Theorem~A. If ${\rm rank}_{\bZ}E(\bQ)=2$ and $\Sha(E/\bQ)[p^\infty]$ is finite, then the generalized Kato classes $\kappa_{\al,\al^{-1}}$ and $\kappa_{\beta,\beta^{-1}}$ are both non-zero and generate the strict Selmer group ${\rm Sel}_{\rm str}(\bQ,V_pE)$. 
\end{mainthmB}

The above corollary has the flavor of a rank $2$ analogue of Skinner's converse to Kolyvagin's theorem \cite{skinner}. In the opposite direction, Theorem~A 
also yields the following rank $2$ analogue of Kolyvagin's theorem in terms of generalized Kato classes. 

\begin{mainthmC}Let the hypotheses be as in Theorem~A. Then the implication 
\[\kappa_{\al,\al^{-1}}\neq 0\implies\dim_{\Qp}{\rm Sel}(\bQ,V_pE)=2\]
holds.
\end{mainthmC}

\subsection{Outline of the proofs}
We conclude the Introduction with a sketch of the proof of the implication ${\rm (ii)}\Rightarrow{\rm (i)}$ in Theorem~A, establishing the non-vanishing of 
\[
\kappa_{E,K}:=\kappa_{\alpha,\alpha^{-1}}\in\rmH^1(\bQ,V_pE).
\] 

\begin{itemize}
	\item{}\emph{Step 1: Euler system construction of Bertolini--Darmon theta elements.}
\end{itemize}

Denote by $\Gamma_\infty$ the Galois group of the anticyclotomic $\bZ_p$-extension of $K$. Building on 
generalizations of Gross' explicit form of Waldspurger's special value formula \cite{walds,gross7}, one can construct a $p$-adic $L$-function $\Theta_{f/K}\in\bZ_p\dBr{\Gamma_\infty}$ interpolating ``square-roots'' of the central critical values $L(E/K,\phi,1)$, as $\phi$ runs over finite order characters of $\Gamma_\infty$ (see \cite{BDmumford-tate,ChHs1}). 
%
%
The element $\Theta_{f/K}$ 
has been widely studied in the literature, 
but its place in Perrin-Riou's vision \cite{PR:Lp,LZ2}, 
whereby $p$-adic $L$-functions ought to arise as the image of 
families of special cohomology classes under generalized Coleman power series maps, remained mysterious.  

Letting $\kappa(f,\bfg\bfh)=\{\kappa(f,\bfg_\ell,\bfh_\ell)\}_\ell$ be the $p$-adic family of diagonal cycle classes giving rise to $\kappa(f,g_{\alpha},h_{\alpha^{-1}})$ in the limit at $\ell\to 1$, in $\S\ref{sec:GKC}$ we prove that   
\begin{equation}\label{eq:big-ERL}
{\rm Col}^\eta({\rm loc}_p(\kappa(f,\bfg\bfh))=\Theta_{f/K}\cdot(\textrm{nonzero constant}),
\end{equation}
where ${\rm Col}^\eta$ is a generalized Coleman power series map defined in terms of an anticyclotomic variant of Perrin-Riou's big exponential map. 
The proof of $(\ref{eq:big-ERL})$ combines a refinement of the explicit reciprocity law of Darmon--Rotger \cite{DR2} with a factorization of the $p$-adic triple product $L$-function \cite{hsieh-triple}. 

\begin{itemize}
	\item{}\emph{Step 2: Leading coefficient formula and derived $p$-adic heights.} 
\end{itemize}

Viewing $(\ref{eq:big-ERL})$ as an identity in the power series ring $\bZ_p\dBr{T}\iso\bZ_p\dBr{\Gamma_\infty}$, its value at $T=0$ recovers the implication  
\[
L(E,1)=0\quad\Longrightarrow\quad\kappa_{E,K}\in{\rm Sel}(\bQ,V_pE).
\]
To further 
deduce 
the non-vanishing of $\kappa_{E,K}$ we 
consider the leading coefficient of $(\ref{eq:big-ERL})$ at $T=0$. To that end, let 
\begin{equation}\label{eq:intro-fil}
{\rm Sel}(K,V_pE)=S^{(1)}\supset S^{(2)}\supset\cdots\supset S^{(r)}\supset\cdots\supset S^{(\infty)}
\end{equation}
be the filtration defined by Bertolini--Darmon \cite{bd-117} and Howard \cite{howard-derived}, and the associated derived anticyclotomic $p$-adic height pairings 
\[
h^{(r)}\colon S^{(r)}\times S^{(r)}\rightarrow\Qp.
\] 
From the standard properties of $h^{(r)}$, one can easily see that if
\[
{\rm Sel}(\Q,V_pE)={\rm Sel}(K,V_pE)\quad\textrm{and}\quad\dim_{\Qp}{\rm Sel}(\Q,V_pE)=2,
\] 
as we have under the hypotheses of Theorem~A, the filtration $(\ref{eq:intro-fil})$ reduces to
\[
{\rm Sel}(\Q,V_pE)=S^{(1)}=S^{(2)}\cdots =S^{(r)}\quad\textrm{and}\quad S^{(r+1)}=S^{(r+2)}=\cdots =S^{(\infty)}=\stt{0}
\] 
for some 
$r\geqslant 2$. 
Setting 
\[
\rho:=\Ord_{T=0}\Theta_{f/K}(T),
\] 
one can deduce that $r\geqslant\rho$ from the work of Skinner--Urban \cite{SU}; in particular, ${\rm Sel}(\Q,V_pE)=S^{(\rho)}$. 
Based on the explicit Rubin-style formula for derived $p$-adic heights 
established in $\S{3}$, we prove that 
for any basis $(P,Q)$ of ${\rm Sel}(\bQ,V_pE)$, the $\rho$-th derived $p$-adic height $h^{(\rho)}(P,Q)$ is non-zero, and 
\[\kappa_{E,K}=\frac{\bar\theta_{f/K}}{h^{(\rho)}(P,Q)}\cdot\left(P\otimes \log_pQ-Q\otimes \log_pP\right)\cdot(\textrm{explicit nonzero constant in $\overline{\Q}$}),\]
where $\bar\theta_{f/K}:=\left(\frac{d}{dT}\right)^\rho\Theta_{f/K}(T)|_{T=0}$ is the leading coefficient of $\Theta_{f/K}$ (see \corref{C:51.M} for the precise statement).
This in particular implies the non-vanishing of $\kappa_{E,K}$. 

\sk

{\bf Acknowledgements.} We would like to thank John Coates, Dick Gross, and Barry Mazur for their comments on earlier drafts of this paper.


\def\Bd{\boldsymbol d}
\def\rmd{{\rm d}}
\newcommand\unip[1]{u(#1)}
\newcommand\aone[1]{\pDII{#1}{1}}
\renewcommand\cond[1]{c(#1)}

\newcommand{\wtd}{\widetilde}
\newcommand{\Dmd}[1]{\langle #1\rangle}
\newcommand{\beq}{\begin{equation}}
\newcommand{\eeq}{\end{equation}}
\def\isoto{\stackrel{\sim}{\to}}
\newcommand{\wh}{\widehat}
\def\hookto{\hookrightarrow}
\def\longto{\longrightarrow}
\def\parref#1{\ref{#1}}
\def\thmref#1{Theorem~\parref{#1}}
\def\expectthmref#1{Expected theorem~\parref{#1}}
\def\propref#1{Proposition~\parref{#1}}
\def\corref#1{Corollary~\parref{#1}}     \def\remref#1{Remark~\parref{#1}}
\def\secref#1{\S\parref{#1}}
\def\lmref#1{Lemma~\parref{#1}}

\def\cH{{\mathcal H}}
\def\uf{\varpi}
\def\bfC{\mathbf C}
\def\frakg{{\mathfrak g}}
\def\frakL{{\mathfrak L}}
\def\imply{\Rightarrow}
\def\Spec{\mathrm{Spec}\,}
\def\padic{$p$-adic}
\def\bfz{{\mathbf z}}
\def\bfw{{\mathbf w}}
\def\bfu{{\mathbf u}}
\def\bfG{{\mathbf G}}
\def\ulQ{{\ul{Q}}}
\def\ulT{\ul{T}}
\def\ulk{\ul{k}}
\def\Qtn{D^\x}
\def\Om{\boldsymbol \omega}
\def\Qx{{Q_1}}
\def\Qy{{Q_2}}
\def\Qz{{Q_3}}
\def\unb{{\rm unb}}
\def\bal{{\rm bal}}
\def\qme{q}
\def\Tau{\boldsymbol\tau}
\def\Eta{\boldsymbol\eta}
\def\addchar{\boldsymbol\psi}
\def\newW{W}
\def\bfL{\bft}
\def\bfa{\mathbf a}
\def\ot{\otimes}
\def\bdsH{{\boldsymbol H}}
\def\bdsf{{\boldsymbol f}}
\def\bdsg{{\boldsymbol g}}
\def\bdsh{{\boldsymbol h}}
\def\cyc{\boldsymbol\varepsilon_{\rm cyc}}
\def\Zp{\Z_p}
\def\rmH{{\rm H}}
\def\IwH{\wh\rmH}
\def\opcpt{U}
\def\Iw{{\rm Iw}}
\def\dR{{\rm dR}}
\def\Tw{{\rm Tw}}
\def\wG{\widetilde{G}}
\def\cl{{\rm cl}}
\def\rig{{\rm rig}}
\def\cris{{\rm cris}}
\def\loc{{\rm loc}}
\def\bfD{{\mathbf D}}
\def\sD{{\mathscr D}}
\def\bfx{{\mathbf x}}
\def\sH{{\mathscr H}}
\def\Exp{\Omega}
\def\cyc{\varepsilon_\cF}
\def\gen{{\boldsymbol \gamma}}
\def\Gm{\wh\bfG_m}
\def\Sel{{\rm Sel}}
\def\frakm{\mathfrak m}
\def\Mod{{\bf Mod}}
\def\om{g}
\def\bfe{\boldsymbol{e}}
\def\COL{{\rm Col}}
\def\ur{{\rm ur}}
\def\base{K}
\def\frakp{{\mathfrak p}}
\def\frakP{{\mathfrak P}}
\newcommand\ol[1]{\overline{#1}}
\newcommand{\bpair}[2]{\left[ #1, #2\right]}
\def\Proj{{\rm pr}}
\def\Qpur{\cQ}
\def\Cp{\bC_p}
\def\et{\rm et}
\def\cris{{\rm cris}}
\def\Fr{\mathrm{Fr}}
\section{Derived $p$-adic height pairings}

In this section, we review the definition of the derived $p$-adic height pairings in \cite{howard-derived}. 

\subsection{Notation and definitions}
Let $p$ be a prime, let $\base$ be a number field, and let $\Sigma$ be a finite set of places of $\base$ containing the archimedean places and the places above $p$. Let $\base_\Sigma$ be the maximal algebraic extension of $\base$ unramified outside $\Sigma$ and set $G_{\base,\Sigma}=\Gal(\base_\Sigma/\base)$. Let $\base_\infty/\base$ be a $\Zp$-extension in $\base_\Sigma$, and assume that all primes above $p$ are totally ramified in $\base_\infty$. Denote by $\base_n$ be the subfield with of $K_\infty$ with $[\base_n\colon \base]=p^n$, and set $\Gamma_n=\Gal(K_n/K)$ and  $\Gamma_\infty=\Gal(K_\infty/K)$. Let $\Lam=\Zp\dBr{\Gamma_\infty}$ and let $\kappa_\Lam\colon G_{\base,\Sigma}\to \Gal(\base_\infty/\base)\to \Lam^\x$ be the tautological character 
$\kappa_\Lam(\sigma)=\sigma|_{\base_\infty}$. Let $\iota\colon \Gamma_\infty\to \Gamma_\infty$ be the involution $\gamma\mapsto \gamma^{-1}$, and for any $\Lam$-module $M$ and $k\in\mathbf{Z}$, let $M\{k\}$ be the $G_{\base,\Sigma}$-module $M$ on which $G_{\base,\Sigma}$ acts via $\kappa_\Lam^k$. 

Let $\cO$ be a local ring finitely generated over $\Zp$, let $\frakm$ be the maximal ideal of $\cO$, and put $\Lam_\cO=\Lam\ot_{\Zp}\cO$. Denote by $\Mod_\cO$ the category of $\cO[G_{\base,\Sigma}]$-modules finite free over $\cO$. For $T$ an object of $\Mod_\cO$ we let $T_{\Lam}=T\ot_\cO\Lam_\cO\{-1\}$ be the $G_{\base,\Sigma}$-module $T\ot_\cO\Lambda_\cO$ twisted by $\kappa_\Lam^{-1}$. Let $\cK$ be the localization of $\Lam_\cO$ at the prime $\frakm\Lam_\cO$, and set $P=\cK/\Lam_\cO$. Likewise we define $T_\cK=T\ot_\cO\cK\{-1\}$ and $T_P=T\ot_{\cO} P\{-1\}$. We shall denote the limits 
\[
\rmH^1(\base_\infty,T):=\dirlim_n\rmH^1(\base_n,T),\quad\quad\wh\rmH^1(\base_\infty,T):=\prolim_n \rmH^1(\base_n,T)
\] 
with respect to the restriction and corestriction maps, respectively. Let \[\Proj_{K_n}\colon \wh\rmH^1(K_\infty,T)\to \rmH^1(K_n,T)\] be the canonical projection map. Throughout we shall make use of the identification
\[
\rmH^1(\base,T_\Lam)=\wh\rmH^1(\base_\infty,T) 
\]
deduced from \cite[Lem.~1.4]{howard-derived} and Shapiro's lemma. 
Let $T^*=\Hom(T,\cO(1))$ 
and denote by $e\colon T\times T^*\to \cO(1)$ the canonical $G_{\base,\Sigma}$-equivariant perfect paring, which uniquely extends to a perfect $G_{\base,\Sigma}$-equivariant pairing
\[
e_\Lam\colon T_\Lam\times T^*_\Lam\to\Lam_\cO(1)
\]
characterized by 
\[
e_\Lam(t\ot \lam_1,s\ot \lam_2)=\lam_1\lam_2^\iota e_\Lam(t,s)\quad
\]
for all $\lam_1,\lam_2\in\Lam_\cO$. 

For any place $v$ of $K$ and any finite extension $L$ of $K_v$, let $\inv_L:\rmH^2(L,\cO(1))\iso \cO$ be the invariant map and let 
$\pairing_{L}\colon \rmH^1(L,T)\times\rmH^1(L,T^*)\to\cO$ be the perfect pairing $\pair{z}{w}_L:=\inv_L(z\cup w)$, and define the bilinear pairing 
\begin{align*}\pairing_{\base_\infty,v}\colon \rmH^1(\base_v,T_\Lam)\times \rmH^1(\base_v,T^*_\Lam)&\to \rmH^2(\base_v,\Lam_\cO(1))\iso\Lam_\cO
\end{align*}
by $\pair{\bfz}{\bfw}_{\base_\infty,v}=\inv_v(e_\Lam(\bfz\cup\bfw))$. 
Fix a topological generator $\gen$ of $\Gamma_\infty$ and set
\[
\om_n:=\boldsymbol{\gamma}^{p^n}-1\in\Lam.
\]
Thus $\Lam_\cO/(g_n)\iso\cO[\Gamma_n]$, and if $\bfz=(\bfz_n)\in\rmH^1(K_v,T_\Lam)=\prolim_n\rmH^1(\base_{n,v},T)$ and $\bfw=(\bfw_n)\in\rmH^1(\base_v,T_\Lam^*)=\prolim_n\rmH^1(\base_{n,v},T^*)$, then  
\beq\label{E:pairing1.H}\pair{\bfz}{\bfw}_{\base_{\infty,v}}\pmod{\om_n}=\sum_{\tau\in\Gamma_n}\pair{\bfz_n^{\tau^{-1}}}{\bfw_n}_{\base_{n,v}}\tau.\eeq

Let $\cF=\stt{\rmH^1_\cF(\base_v,T_\cK)}_{v\in\Sigma}$ be a Selmer structure on $T_\cK$, namely a choice of $\cK$-submodule $\rmH^1_\cF(\base_v,T_\cK)\subset\rmH^1(\base_v,T_\cK)$ for every $v\in\Sigma$, and let $\rmH^1_\cF(\base_v,T_P)$ be the image of the natural  map $\rmH^1_\cF(\base_v,T_\cK)\to \rmH^1(\base_v,T_P)$ induced by the quotient $\cK\to P$.  Define the Selmer module $\rmH^1_\cF(\base,T_P)$ to be the kernel of the map
\[
\rmH^1(G_{\base,\Sigma},T_P)\to \prod_{v\in\Sigma}\rmH^1(\base_v,T_P)/H^1_\cF(\base_v,T_P).
\]

\subsection{Abstract Rubin formula}
\label{subsec:height}

In this subsection, we suppose that $\frakm^m=0$ for some $m>0$, namely that $\cO$ is Artinian. By \cite[Lem.~1.2]{howard-derived}, we then have
\[
K=\bigcup_{n=0}^\infty\Lam_\cO\frac{1}{\om_n}.
\]   
Moreover, by \cite[Lem.~1.5]{howard-derived}  and Shapiro's lemma, there is a natural isomorphism
\beq\label{E:iso1.H}\eta_{\gen}\colon \rmH^1(\base,T_P)=\dirlim_{n}\rmH^1(\base,T_\Lam\ot\Lam_\cO\om_n^{-1}/\Lam_\cO)\iso\dirlim_n\rmH^1(\base_n,T_\Lam/\om_n T_\Lam)=\rmH^1(\base_\infty,T).\eeq
By definition, for  $\bfz=\stt{\bfz_n}\in\wh\rmH^1(\base_\infty,T)$ we have 
\beq\label{E:iso2.H}\eta_\gen(\frac{\bfz}{\gen-1})=\Proj_\base(\bfz)\in\rmH^1(\base,T).\eeq
For each $n$, let $\rmH^1_\cF(\base_n,T)$ be the Selmer module consisting of classes $s\in \rmH^1(\base_n,T)$ such that the image of $s$ in $\rmH^1(\base_\infty,T)$ belongs to $\eta_\gen(\rmH^1_\cF(\base,T_P))$. Thus \[\rmH^1_\cF(\base_\infty,T)=\dirlim_n\rmH^1_\cF(\base_n,T)=\eta_\gen(\rmH^1_\cF(\base,T_P)).\] Let $J$ be the augmentation ideal of $\Lam_\cO$, i.e.,   
the principal ideal of $\Lam_\cO$ generated by $\gen-1$, and  
for $r>0$ put
\beq\label{E:1.H}Y_T^{(r)}:=\rmH^1_\cF(\base_\infty,T)[J]\cap J^{r-1}\rmH^1_\cF(\base_\infty,T).\eeq
This defines a decreasing filtration $Y_T^{(1)}\supset Y_T^{(2)}\supset Y_T^{(2)}\supset\cdots$.

Let $\cF^\perp=\stt{\rmH^1_{\cF^\perp}(\base_v,T^*_K)}_{v\in\Sigma}$ be the Selmer structure on $T^*_K$ such that $\rmH^1_{\cF^\perp}(\base_v,T^*_K)$ and $\rmH^1_\cF(\base_v,T_K)$ are orthogonal complements under local Tate duality for every $v\in\Sigma$, and let
\[
[-,-]_{\rm CT}\colon \rmH^1_\cF(\base,T_P)\times \rmH^1_{\cF^\perp}(\base,T^*_P)\to P
\]
be the $\Lam_\cO$-adic Cassels--Tate pairing of \cite[Thm.~1.8]{howard-derived}. The \emph{$r$-th derived height pairing} 
\[
h_\cO^{(r)}(-,-)\colon Y_T^{(r)}\times Y_{T^*}^{(r)}\to J^r/J^{r+1}
\]
in \cite[Def.~2.2]{howard-derived} is defined by 
\begin{align*}
h_\cO^{(r)}(z,w):=&(\gen-1)^2\cdot [\eta_{\gen}^{-1}(z),\eta_{\gen}^{-1}(w)]_{\rm CT}\\
=&(\gen-1)^{r+1}\cdot[\eta_{\gen}^{-1}(u),\eta_{\gen}^{-1}(w)]_{\rm CT},
\end{align*}
writing $z=(\gen-1)^{r-1}u$ with $u\in\rmH^1_\cF(\base,T_P)$. 
Note that $[\eta_{\gen}^{-1}(u),\eta_{\gen}^{-1}(w)]_{\rm CT}\in (\gen-1)^{-1}\Lam/\Lam$, so that $h_\cO^{(r)}(z,w)$ belongs to $J^{r}/J^{r+1}$. 

The following is a restatement of \cite[Thm.~2.5]{howard-derived}, which can be viewed as an abstract generalization of Rubin's formula \cite[Thm. 3.2(ii)]{rubin-ht} (\emph{cf.} \cite[Prop.~11.5.11]{nekovar310}). 
\begin{prop}\label{P:12.H}
Let $z\in Y_T^{(r)}$ and $w\in Y_{T^*}^{(r)}$. Suppose that there exist $\bfz\in \rmH^1(\base,T_\Lam)$ and $\bfw_\Sigma=(\bfw_v)\in\bigoplus_{v\in\Sigma}\rmH^1_{\cF^\perp}(\base_v,T^*_\Lam)$ such that $\Proj_\base(\bfz)=z$ and $\Proj_{\base_v}(\bfw_v)=\loc_v(w)$. Then \[
h_\cO^{(r)}(z,w)=-\sum_{v\in\Sigma}\pair{\bfz}{\bfw_v}_{\base_{\infty,v}}\pmod{J^{r+1}}.
\]
\end{prop}

\begin{proof}Let $y=\eta_\gen^{-1}(z)\in \rmH^1_\cF(\base,T_P)$ and $t=\eta_\gen^{-1}(w)\in\rmH^1_{\cF^\perp}(\base,T^*_P)$. Choose cochains $\wtd y\in C^1(G_{\base,\Sigma},T_\cK)$ and $\wtd t\in C^1(G_{\base,\Sigma},T^*_\cK)$ lifting $s$ and $t$, respectively, let $\ep_0\in C^2(G_{\base,\Sigma},P(1))$ be such that $d\ep_0=d\wtd y\cup \wtd t$, and choose $\wtd t_\Sigma\in \bigoplus_{v\in\Sigma}Z^1(G_{\base_v},T^*_\cK)$ lifting $\loc_\Sigma(\wtd t)\in\bigoplus_{v\in\Sigma}Z^1(\base_v,T^*_P)$. According to the definition of the Cassels--Tate pairing \cite[(2), p.~1321]{howard-derived}, we find that
\beq\label{E:CT.H}\begin{aligned}
h_\cO^{(r)}(z,w)&=(\gen-1)^2\cdot [y,t]_{\rm CT}
=(\gen-1)^2\cdot \inv_\Sigma(\loc_{\Sigma}(\wtd y)\cup \wtd t_\Sigma-\loc_\Sigma(\ep_0)).\end{aligned}
\eeq

Let $\wtd\bfz\in Z^1(G_{\base,\Sigma},T_{\Lam})$ and $\wtd\bfw_\Sigma\in \bigoplus_{v\in\Sigma}Z^1(G_{\base_v},T^*_\Lam)$  be cocycles representing $\bfz$ and $\bfw_\Sigma$. 
Then $\wtd y=\wtd\bfz/(\gen-1)$ and $\wtd t_\Sigma=\wtd\bfw_{\Sigma}/(\gen-1)$ are liftings of $z$ and $t_\Sigma$, and 
using \eqref{E:CT.H} with $\ep_0=0$ (note that $d\wtd \bfz=0$), we obtain
\begin{align*}h_\cO^{(r)}(z,w)
&=(\gen-1)^2\cdot\inv_{\Sigma}(e_\Lam(\frac{\loc_\Sigma(\wtd\bfz)}{\gen-1}\cup \frac{\wtd\bfw_\Sigma}{\gen-1}))\\
&=-\inv_{\Sigma}(e_\Lam(\loc_\Sigma(\bfz)\cup \bfw_\Sigma))=-\sum_{v\in\Sigma}\pair{\bfz}{\bfw_v}_{\base_{\infty,v}}\pmod{J^{r+1}}.\end{align*}
This completes the proof.
\end{proof}

\subsection{Derived $p$-adic heights for elliptic curves}\label{SS:derived}

Let $E$ be an elliptic curve over $\base$ with good ordinary reduction at every place above $p$. Let $T=\varprojlim_kE[p^k]$ be the $p$-adic Tate module of $E$, and take $\Sigma$ to consist of the archimedean places, the places above $p$, and the places of bad reduction of $E$. Let $T_k=E[p^k]$, and consider the modules $Y_{T_k}^{(r)}$ 
defined in \eqref{E:1.H} taking for $\cF$ the Selmer structure in \cite[Def.~3.2]{howard-derived}. Since $T_k^*=T_k$ and $\cF^\perp=\cF$ by the Weil pairing, the discussion of $\S\ref{subsec:height}$ yields a derived height pairing $h_{\Z/p^k\Z}^{(r)}$ on $Y_{T_k}^{(r)}\times Y_{T_k}^{(r)}$. 
The constructions of $Y_{T_k}^{(r)}$ and $h_{\Z/p^k\Z}$ are clearly compatible  
under the quotient map $\Z/p^{k+1}\Z\to \Z/p^k\Z$, and in the limit they define
\[
Y_T^{(r)}:=\prolim_k Y_{T_k}^{(r)},\quad h^{(r)}:=\prolim_k h_{\Z/p^k\Z}^{(r)}.
\]

According to \cite[Lem.~4.1]{howard-derived} there is canonical isomorphism
\begin{equation}\label{eq:how-iso}
Y_T^{(1)}\ot_{\Zp}\Qp\simeq{\rm Sel}(\base,V_pE).
\end{equation}
Letting $S_p^{(r)}(E/\base)$ be the subspace of ${\rm Sel}(\base,V_pE)$ spanned by the image of $Y_T^{(r)}\subset Y_T^{(1)}$ under the  isomorphism $(\ref{eq:how-iso})$, we have $S_p^{(1)}(E/\base)={\rm Sel}(\base,V_pE)$ and for every $r>0$ we obtain the $r$-th derived $p$-adic height pairing 
\[h^{(r)}: S_p^{(r)}(E/\base)\times S_p^{(r)}(E/\base)\to J^r/J^{r+1}\ot_{\Zp}\Qp,\] 
where $J$ is the augmentation ideal of $\Lam=\Zp\dBr{\Gamma_\infty}$.

\begin{cor}\label{C:13.H}
Let $z,w\in S_p^{(r)}(E/\base)$. Suppose that there exist a global class $\bfz\in \wh\rmH^1(\base_\infty,T)$ and local classes $\bfw_v\in\prolim_n\rmH^1_{\rm fin}(\base_{n,v},T)$ for every $v\mid p$ such that $\Proj_\base(\bfz)=z$ and $\Proj_{\base_v}(\bfw_{v})=\loc_v(w)$. Then 
\[
h^{(r)}(z,w)=-\sum_{v\mid p}\pair{\loc_v(\bfz)}{\bfw_v}_{\base_{\infty,v}}\pmod{J^{r+1}}.
\]
\end{cor}
\begin{proof} This follows from \propref{P:12.H} and the fact that $\rmH^1(\base_{n,v},T)\ot\Qp=0$ for $v\ndivides p$.
\end{proof}

\section{Perrin-Riou's theory for Lubin--Tate formal groups}

In this section we explicitly compute the derived $p$-adic height pairings for elliptic curves via Perrin-Riou's big exponential maps.

\subsection{Preliminaries}

We begin by reviewing the generalization of Perrin-Riou's theory \cite{PR115} to Lubin--Tate formal groups developed in \cite{kobayashi-pr}. Throughout we fix a completed algebraic closure $\Cp$ of $\Qp$. Let $\Qp^\ur\subset \Cp$ be the maximal unramified extension of $\Qp$ and let $\Fr\in\Gal(\Qp^\ur/\Qp)$ be the absolute Frobenius. Let $F/\Qp$ be a finite unramified extension and let $\cO=\cO_F$ be the valuation ring of $F$. Put
\[R:=\cO\dBr{X}.\] 
 
Let $\cF=\Spf R$ be a relative Lubin--Tate formal group of height one defined over $\cO$, and for each $n\in\Z$ set $\cF^{(n)}:=\cF\times_{\Spec\cO,\Fr^{-n}}\Spec\cO$. The Frobenius morphism $\varphi_\cF\in\Hom(\cF,\cF^{(-1)})$ induces a homomorphism $\varphi_\cF\colon R\to R$ defined by \[\varphi_\cF(f):=f^{\Fr}\circ\varphi_\cF,\] 
where $f^{\Fr}$ is the conjugate of $f$ by $\Fr$. Let $\psi_\cF$ be the left inverse of $\varphi_\cF$ 
satisfying 
\[
\varphi_\cF\psi_\cF (f)=p^{-1}\sum_{x\in \cF[p]}f(X\oplus_\cF x).
\] 
Let $F_\infty=\bigcup_{n=1}^\infty F(\cF[p^n])$ be the Lubin--Tate $\Zp^\x$-extension associated with the formal group $\cF$, and for every $n\geqslant -1$, let $F_n$ be the subfield of $F_\infty$ with $\Gal(F_n/F)\iso(\Z/p^{n+1}\Z)^\times$ (so $F_{-1}=F$). Letting $G_\infty=\Gal(F_\infty/F)$, there is a unique decomposition $G_\infty=\Delta\times \Gamma^\cF_\infty$, where $\Delta\iso\Gal(F_0/F)$ is the torsion subgroup of $G_\infty$ and $\Gamma^\cF_\infty\iso \Zp$. 

For every $a\in\Zp^\x$, there is a unique formal power series $[a](X)\in R$ such that 
\[
[a]^{\Fr}\circ\varphi_\cF =\varphi_\cF\circ [a]\quad\textrm{and}\quad
[a](X)\con aX\pmod{X^2}. 
\]
Letting $\varepsilon_\cF\colon G_\infty\isoto\Zp^\x$ be the Lubin--Tate character, we let $\sigma\in G_\infty$ act on $R$ by 
\[
\sigma\cdot f(X):=f([\varepsilon_\cF(\sigma)](X)),
\] 
thus making $R$ into an $\cO\dBr{G_\infty}$-module.

\begin{lm}
	\label{lem:rk1}
	$R^{\psi_\cF=0}$ is free of rank one over $\cO\dBr{G_\infty}$.
\end{lm}
\begin{proof}This is a standard fact. See \cite[Prop.~5.4]{kobayashi-pr}.\end{proof} 

Let $L\subset \Cp$ be a finite extension over $\Qp$, and let $V$ be a finite-dimensional $L$-vector space on which $G_{\Qp}$ acts as a continuous $L$-linear crystalline representation. Let $\bfD(V)=\bfD_{{\rm cris},\Qp}(V)$ be the filtered $\varphi$-module associated with $V$ over $\Qp$ and set
\[\sD_\infty(V):=\bfD(V)\ot_{\Zp}R^{\psi_\cF=0}\iso \bfD(V)\ot_{\Zp}\cO\dBr{G_\infty}.\] 
Let $d:R\to\Omega_R$ be the standard derivation. Fix an invariant differential $\omega_\cF\in\Omega_R$, and let $\log_\cF\in R\wh \ot\Qp$ the logarithm map satisfying $\log_\cF(0)=0$ and $d\log_\cF=\omega_\cF$. Let also $\partial\colon R\to R$ be defined by $df=\partial f\cdot \omega_\cF$. 

Let $\epsilon=(\epsilon_n)\in T_p\cF=\prolim_n\cF^{(n+1)}[p^{n+1}]$  
be a basis of the $p$-adic Tate module of $\cF$, where the inverse limit is with respect to the maps $\varphi^{\Fr^{-(n+1)}}\colon \cF^{(n+1)}[p^{n+1}]\to \cF^{(n)}[p^{n}]$. 
Following \cite[p.~42]{kobayashi-pr}, we associate to $\ep$ and $\omega_\cF$ a $p$-adic period $t_\epsilon\in B_{\cris}^+$ for $\cF$ as follows. For each $n$, there exists a unique isomorphism $\varphi^\flat_n:\cF^{(n)}\to\cF$ such that 
\[
\varphi^{\Fr^{-1}}\circ\cdots\circ\varphi^{\Fr^{-(n-1)}}\circ\varphi^{\Fr^{-n}}=[p^n]\circ\varphi^\flat_n.
\] 
Put $w_n:=\varphi^\flat_n(\ep_{n-1})\in\cF[p^n]$, so that $[p](w_n)=w_{n-1}$ by definition. Let $A_{\rm inf}=A_{\rm inf}(\cO_{\Cp}/\cO_F)$ and $\theta:A_{\rm inf}\to\cO_{\Cp}$ be 
as defined in \cite[\S 1.2.2]{fontaine223}. It is not difficult to show that there is a unique sequence $(\wtd w_n)$ of elements in $\cF(A_{\rm inf})$ such that $[p](\wtd w_n)=\wtd w_{n-1}$ and $\theta(\wtd w_n)=w_n$, and we set $t_\ep:=\log_\cF(\wtd w_0)\in B_{\cris}^+.$ This $p$-adic period $t_\ep$ satisfies
\beq\label{E:unif.H}\bfD_{{\rm cris},F}(\varepsilon_\cF)=Ft_\ep^{-1}, \quad\varphi t_\ep=\varpi t_\ep,\eeq where $\varpi$ is the uniformizer in $F$ such that $\varphi_\cF^*(\omega_\cF^{\Fr})=\varpi\cdot\omega_\cF$. 

Fix an extension $\breve\varepsilon_\cF\colon \Gal(F_\infty/\Qp)\to L^\x$ of the Lubin--Tate character $\varepsilon_\cF$, and for each $j\in\Z$  let $V\Dmd{j}:=V\ot_{L}\breve\varepsilon_\cF^j$ denote the $j$-th Lubin--Tate twist of $V$. By definition, $\bfD_{{\rm cris},F}(V\Dmd{j})=\bfD(V)\ot_{\Qp} Ft_\ep^{-j}$. Define the derivation ${\rm d}_\ep:\sD(V\Dmd{j})\to \sD(V\Dmd{j-1})$ by 
\[ 
{\rm d}_\ep f:=\eta t_\ep\ot \partial g,\quad\textrm{where $f=\eta\ot g\in \bfD_{\cris,F}(V\Dmd{j})\ot_{\cO}R^{\psi_{\cF}=0}$},
\] 
and the map 
\[
\widetilde{\Delta}\colon\sD_\infty(V)\to \bigoplus_{j\in\Z}\frac{\bfD_{{\rm cris},F}(V\Dmd{-j})}{1-\varphi}
\]
by $f\mapsto (\partial^jf(0)t_\ep^j\pmod{(1-\varphi}))$.

\begin{rem}
When $\cF=\Gm$, we have $F_\infty=F(\zeta_{p^\infty})$, the corresponding Lubin--Tate character is the $p$-adic cyclotomic character $\varepsilon_{\rm cyc}:G_{\Qp}\to\Zp^\x$, $\varphi_{\Gm}(f)=f^{\Fr}((1+X)^p-1)$, and $\psi_{\Gm}(f)$ is given by the unique power series such that 
\[\varphi_{\Gm}\psi_{\Gm}(f)=p^{-1}\sum_{\zeta^p=1}f(\zeta(1+X)-1).\]
If we take $\omega_{\Gm}$ to be the invariant differential $(1+X)^{-1}dX$, then $\partial=(1+X)\frac{d}{dX}$ and $\log_{\Gm}$ is the usual logarithm $\log(1+X)$. In the following, we fix a sequence $\stt{\zeta_{p^n}}_{n=1,2,3,\ldots}$ of primitive $p^n$-th roots of unity  with $\zeta_{p^{n+1}}^p=\zeta_{p^n}$, and let $t\in \bfD(V)$ be the $p$-adic period corresponding to $(\zeta_{p^{n+1}}-1)\in T_p\Gm$.
\end{rem}

\subsection{Perrin-Riou's big exponential map and the Coleman map}
\label{subsec:Col}

For a finite extension $K$ over $\Qp$, let \[\exp_{K,V}\colon \bfD(V)\ot_{\Qp} K\to \rmH^1(K,V)\] be Bloch--Kato's exponential map \cite[\S 3]{BK}.  In this subsection, we recall the main properties of Perrin-Riou's map $\Omega_{V,h}$ interpolating $\exp_{F,V\Dmd{j}}$ as $j$ runs over non-negative integers $j$. 

Let $V^*:=\Hom_L(V,L(1))$ be the Kummer dual of $V$ and denote by
\[[-,-]_V:\bfD(V^*)\ot K\times \bfD(V)\ot K\to K\ot_{\Qp} L\]
the $K$-linear extension of the de Rham pairing
\[
\pairing_{\dR}\colon\bfD(V^*)\times \bfD(V)\to L.
\] 
Let $\exp^*_{K,V}:\rmH^1(K,V)\to \bfD(V)\ot K$ be the dual exponential map characterized uniquely by 
\[\Tr_{K/\Qp}([x,\exp^*_{K,V}(y)]_V)=\pair{\exp_{K,V^*}(x)}{y}_V,
\]
for all $x\in\bfD(V^*)\ot K$, $y\in\rmH^1(K,V)$.

Choose a $G_{\Qp}$-stable $\cO_L$-lattice $T\subset V$, and let \[\wh\rmH^1(F_\infty,T)=\prolim_{n} \rmH^1(F_n,T),\quad \wh\rmH^1(F_\infty,V)=\wh\rmH^1(F_\infty,T)\ot_{\Zp}\Qp\] be the Iwasawa cohomology $\Zp\dBr{G_\infty}$-modules associated with $V$. We denote by \[\Tw_j:\wh\rmH^1(F_\infty,V)\iso \wh\rmH^1(F_\infty,V\Dmd{j})\] the twisting map by $\breve \varepsilon_\cF^j$. 
For a non-negative real number $r$ and any subfield $K$ in $\Cp$, we put
\[\sH_{r,K}(X)=\stt{\sum_{n\geqslant 0,\tau\in\Delta}c_{n,\tau}\cdot\tau\cdot X^n\in K[\Delta]\dBr{X}\mid \sup_n\abs{c_{n,\tau}}_pn^{-r}<\infty\text{ for all }\tau\in\Delta},\]
where $\abs{\cdot}_p$ is the normalized valuation of $K$ with $\abs{p}_p=p^{-1}$. Let $\gen$ be a topological generator of $\Gamma^\cF_\infty$, and denote by $\sH_{r,K}(G_\infty)$ the ring of elements $\{f(\gen-1)\colon f\in \sH_{r,K}(X)\}$, 
so in particular $\sH_{0,K}(G_\infty)=\cO_K\dBr{G_\infty}\ot\Qp$. Put \[\sH_{\infty,K}(G_\infty)=\bigcup_{r\geqslant 0}\sH_{r,K}(G_\infty).\] 

For $n\geqslant -1$, we define a map \[\Xi_{n,V}\colon \bfD(V)\ot_{\Qp}\sH_{\infty,F}(X)\to  \bfD(V)\ot_{\Qp}  F_n\] by 
\[\Xi_{n,V}(G):=\begin{cases}p^{-(n+1)}\varphi^{-(n+1)}(G^{\Fr^{-(n+1)}}(\epsilon_n))&\text{ if }n\geqslant 0,\\
(1-p^{-1}\varphi^{-1})(G(0))&\text{ if }n=-1.\end{cases}
\]

Let $h$ be a positive integer such that $\bfD(V)=\Fil^{-h}\bfD(V)$ and assume that $\rmH^0(F_\infty,V)=0$. The next result on the construction of the big exponential map and its explicit interpolation formulas is originally due to Perrin-Riou and Colmez for $\cF=\Gm$, and extended by Kobayashi and Shaowei Zhang to general relative Lubin--Tate formal groups of height one. 

\begin{thm}[Perrin-Riou, Colmez, Kobayashi, Shaowei Zhang]\label{T:BigExp} Let $\wtd\Lam:=\Zp\dBr{G_\infty}$. There exists a big exponential map 
	\[\Exp^\epsilon_{V,h}: \sD_\infty(V)^{\widetilde{\Delta}=0}\to \wh\rmH^1(F_\infty,T)\ot_{\wtd\Lam} \sH_{\infty, \Qp}(G_\infty)\]
which is $\wtd\Lam$-linear and characterized by the following interpolation property. Let $g\in \sD_\infty(V)^{\widetilde{\Delta}=0}$.  
If $j\geqslant 1-h$, then 
\[	\Proj_{F_n}(\Tw_j\circ\Exp^\epsilon_{V,h}(g))=(-1)^{h+j-1}(h+j-1)!\cdot  \exp_{F_n,V\Dmd{j}}(\Xi_{n,V\Dmd{j}}(\rmd_\ep^{-j}G))\in \rmH^1(F_n,V\Dmd{j}),
\]
and if $j\leqslant -h$, then 
	\[\exp^*_{F_n,V\Dmd{j}}(\Proj_{F_n}(\Tw_j\circ\Exp^\epsilon_{V,h}(g)))=\frac{1}{(-h-j)!}\cdot \Xi_{n,V\Dmd{j}}(\rmd_\ep^{-j}G))\in \bfD(V\Dmd{j})\ot_{\Qp} F_n,\]
where $G\in \bfD(V)\ot_{\Qp} \sH_{h,F}(X)$ is a solution of the equation
	\[(1-\varphi\ot\varphi_\cF)G=g.\]
Moreover, if $D_{[s]}$ is a $\varphi$-invariant $\Qp$-subspace of $\bfD(V)$ such that all eigenvalues of $\varphi$ on $D_{[s]}$ have $p$-adic valuation $s$, then $\Omega_{V,h}^\ep$ maps $(D_{[s]}\ot R^{\psi=0})^{\wtd\Delta=0}$ into $\wh\rmH^1(F_\infty,T)\ot_{\wtd\Lam}\sH_{s+h,\Qp}(G_\infty)$.
\end{thm}
\begin{proof}
In the case $\cF=\Gm$, the construction of 
$\Omega_{V,h}^\ep$ and its interpolation property at integers $j\geqslant 1-h$ is due to Perrin-Riou \cite[\S 3.2.3 Th\'eor\`eme, \S 3.2.4(i)]{PR115}, while the interpolation formula at integers $j\leqslant -h$ is a consequence of the ``explicit reciprocity formula'' proved by Colmez 
\cite[Thm.~IX.4.5]{colmez148}. 
Their methods can be adapted to general relative Lubin--Tate formal groups of height one after some necessary modifications; the details can be found in \cite[App.]{kobayashi-pr} for the construction of $\Omega_{V,h}^\ep$ and the interpolation at $j\geqslant 1-h$, and in \cite[Thm.~6.2]{zhang04} for the explicit reciprocity formula. \end{proof}

To introduce the Coleman map, we further assume the following hypothesis:
\beq\label{E:Delta0.H}\sD_\infty(V)^{\wtd\Delta=0}=\sD_\infty(V).\eeq
 For simplicity, we shall write $\sH_{K}$ for $\sH_{\infty,K}(G_\infty)$ in the sequel. We let \[\bpair{-}{-}_V\colon \bfD(V^*)\ot_{\Qp} \sH_{F}\times \bfD(V)\ot_{\Qp} \sH_{F}\to L\ot_{\Qp}\sH_{F}\] be the pairing defined by 
\[ 
\bpair{\eta_1\ot \lam_1}{\eta_2\ot \lam_2}_V=\pair{\eta_1}{\eta_2} _{\dR}\ot \lam_1\lam_2^\iota
\] 
for all $\lam_1,\lam_2\in\sH_{F}$. For any $e\in R^{\psi_\cF=0}$ and $\ep$ a generator of $T_p\cF$,  there is unique $\cO_L\dBr{G_\infty}$-linear \emph{Coleman map} $\COL_{e}^\ep\colon \wh\rmH^1(F_\infty,V^*)\to \bfD(V^*)\ot_{\Qp} \sH_{F}$ characterized by 
\beq\label{E:Col2.H}\Tr_{F/\Qp}(\bpair{\COL_{e}^\ep(\bfz)}{\eta}_V)=\pair{\bfz}{\Exp^\ep_{V,h}(\eta\ot e)}_{F_\infty}\in L\ot_{\Qp}\sH_{\Qp}\eeq
for all $\eta\in\bfD(V)$. 

Let $\Qpur$ be the completion of $\Qp^\ur$ in $\Cp$, let $\cW$ be the ring of integers of $\Qpur$, and set  $F_n^\ur=F_n\Qp^\ur$. Let $\sg_0\in \Gal(F_\infty^\ur/\Qp)$ be such that $\sg_0|_{\Qp^\ur}=\Fr$ is the absolute Frobenius. Fix an isomorphism $\rho:\wh\bfG_m \iso \cF$ defined over $\cW$ and let $\rho:\cW\dBr{T}\iso R\ot_\cO \cW$ be the map defined by $\rho(f)=f\circ\rho^{-1}$. Then we have 
\[\varphi_\cF\circ\rho=\rho^{\Fr}\circ\varphi_{\Gm}.\]

Let $\bfe\in R^{\psi_\cF=0}$ be a generator over $\cO\dBr{G_\infty}$ and write $\rho(1+X)=h_{\bfe}\cdot \bfe$ for some $h_{\bfe}\in \cW\dBr{G_\infty}$. This implies that $\bfe(0)\in\cO^\x$. Now we fix $\ep=(\ep_n)_{n=0,1,2,\ldots}$ to be the generator of $T_p\cF$ given by \[\ep_n=\rho^{\Fr^{-(n+1)}}(\zeta_{p^{n+1}}-1)\in\cF^{(n+1)}[p^{n+1}].\] 
Let $\eta\in\bfD(V)$ be such that $\varphi\eta=\alpha\eta$ and of slope $s$ (i.e. $\abs{\alpha}_p=p^{-s}$). For every $\bfz\in\wh\rmH^1(F_\infty,V^*)$, we define
\[\COL^{\eta}(\bfz):=\sum_{j=1}^{[F:\Qp]}\bpair{\COL_{\bfe}^\ep(\bfz^{\sigma_0^{-j}})}{\eta}\cdot  h_{\bfe}\cdot \sigma_0^j\in \sH_{s+h,L\Qpur}(\wtd G_\infty),\]
where $\wtd G_\infty:=\Gal(F_\infty/\Qp)$, and $\bpair{-}{-}\colon \bfD(V^*)\ot \sH_{\Qpur} \times \bfD(V)\to \sH_{L\Qpur}$ is the image of $\bpair{-}{-}_V$ under the natural map $L\ot_{\Qp} \sH_{\Qpur}\to \sH_{L\Qpur}$. 

For any integer $j$, put \[\bfz_{-j,n}:=\Proj_{F_n}(\Tw_{-j}(\bfz))\in\rmH^1(F_n,V^*\Dmd{-j}).\] 
We say that a finite order character $\chi$ of $\wtd G_\infty$ has conductor $p^{n+1}$ if $n$ is the smallest integer $\geqslant -1$ such that  $\chi$ factors through $\Gal(F_n/\Qp)$.
\begin{thm} \label{thm:kobayashi}Suppose that $\Fil^{-1}\bfD(V)=\bfD(V)$ and let $h=1$. Let $\psi$ be a $p$-adic character of $\wtd G_\infty$ such that $\psi=\chi\breve\varepsilon_\cF^j$ with $\chi$ a finite order character of conductor $p^{n+1}$. If $j<0$, then 
	\begin{align*}
	&\COL^{\eta}(\bfz)(\psi)=\frac{(-1)^{j-1}}{(-j-1)!}\\
	&\times\begin{cases}
	\bpair{\log_{F,V^*\Dmd{-j}}\bfz_{-j,n}\ot t^{-j}}{(1-p^{j-1}\varphi^{-1})(1-p^{-j}\varphi)^{-1}\eta}&\text{ if }n=-1,\\[1 em]
	p^{(n+1){(j-1)}}\boldsymbol\tau(\psi)\sum\limits_{\tau\in\Gal(F_n/\Qp)}\chi^{-1}(\tau)\bpair{\log_{F_n,V^*\Dmd{-j}}\bfz^{\tau}_{-j,n}\ot t^{-j}}{\varphi^{-(n+1)}\eta}&\text{ if }n\geqslant 0.
	\end{cases}
	\end{align*}
	If $j\geqslant 0$, then 
	\begin{align*}
	&\COL^{\eta}(\bfz)(\psi)=j!(-1)^j\\
	&\times\begin{cases}
	\bpair{\exp^*_{F,V^*\Dmd{-j}}\bfz_{-j,n}\ot t^{-j}}{(1-p^{j-1}\varphi^{-1})(1-p^{-j}\varphi)^{-1}\eta}&\text{ if }n=-1,\\[1 em]
p^{(n+1){(j-1)}}\boldsymbol\tau(\psi)\sum\limits_{\tau\in\Gal(F_n/\Qp)}\chi^{-1}(\tau)\bpair{\exp^*_{F_n,V^*\Dmd{-j}}\bfz^{\tau}_{-j,n}\ot t^{-j}}{\varphi^{-(n+1)}\eta}&\text{ if }n\geqslant 0.
	\end{cases}
	\end{align*}
	Here 
	$\boldsymbol\tau(\psi)$ is the Gauss sum defined by 
	\[\boldsymbol\tau(\psi):=\sum_{\tau\in\Gal(F^{\ur}_n/F^{\ur})}\psi\varepsilon_{\rm cyc}^{-j}(\tau\sigma_0^{n+1})\zeta_{p^{n+1}}^{\tau\sigma_0^{n+1}}.\]
\end{thm}
\begin{proof}
	This follows from the explicit reciprocity formula in \thmref{T:BigExp} and the computation and  in \cite[Thm.~5.10]{kobayashi-pr} (\emph{cf.} \cite[Thm.~4.15]{LZ2}).
\end{proof}

\subsection{The derived $p$-adic heights and the Coleman map}
\label{subsec:applic}

Let $E$ be an elliptic curve over $\Q$ with good ordinary reduction at $p$, and let $V=T_pE\ot_{\Zp} L$ with $L$ a finite extension of $\bQ_p$. We have $\Fil^{-1}\bfD(V)=\bfD(V)$ and $V^*=V$. Let $\omega_E$ be the N\'eron differential of $E$, regarded as an element in $\bfD(\rmH_{\et}^1(E_{/\Qbar},\Qp))$. 
We fix an embedding $\iota_p\colon \Qbar\hookto \Cp$, and for any subfield $H\subset\Qbar$, let $\hat H$ denote the completion of $\iota_p(H)$ in $\Cp$. 

Let $K$ be an imaginary quadratic field in which $p=\frakp\ol{\frakp}$ splits, with $\frakp$ the prime of $K$ above $p$ induced by $\iota_p$.  
Let $K_\infty$ be the anticyclotomic $\Zp$-extension of $K$, and set $\Gamma_\infty=\Gal(K_\infty/K)$ and $\hat\Gamma_\infty=\Gal(\hat K_\infty/\Qp)$. For any integer $c>0$ let $H_c$ be the ring class field of $K$ of conductor $c$, and choosing $c$ to be prime to $p$, put $F=\hat H_c$. Let $\xi\in K$ be a generator of  $\frakp^{[F:\Qp]}$ and let $F_\infty$ be the Lubin--Tate $\Zp$-extension over $F$ associated with $\xi/\ol{\xi}$. By \cite[Prop.~3.7]{kobayashi-pr} we have $F_\infty=\bigcup_{n=0}^\infty\hat H_{cp^n}$, and hence $F_\infty$ is a finite extension of $\hat K_{\infty}$. Moreover, hypothesis \eqref{E:Delta0.H} holds since $\bfD(V)^{\varphi^{[F:\Qp]}=(\xi/\ol{\xi})^j}=\stt{0}$ for any $j\in\mathbf{Z}$, given that the $\varphi$-eigenvalues of $\bfD(V)$ are $p$-Weil numbers while $\xi/\ol{\xi}$ is a $1$-Weil number. 

Let $\alpha_p\in\Zp^\x$ be the $p$-adic unit eigenvalue of the Frobenius map $\varphi$ acting on $\bfD(V)$, and let $\eta\in \bfD(V)=\bfD(\rmH^1_{\et}(E_{/\Qbar},\Qp))\ot\bfD(L(1))$ be a $\varphi$-eigenvector of slope $-1$ such that 
\[\varphi \eta=p^{-1}\alpha_p\cdot\eta\quad\textrm{and}\quad \pair{\eta}{\omega_E\ot t^{-1}}_{\rm dR}=1. 
\]
Let $\bfe\in R^{\psi_\cF=0}$ be a generator over $\cO_F\dBr{G_\infty}$ such that $\bfe(0)=1$. Applying the big exponential map $\Omega_{V,1}^\ep$ in \thmref{T:BigExp}, we define
\beq\label{E:x.H}\bfw^\eta=\Exp^\ep_{V,1}(\eta\ot \bfe)\in \wh\rmH^1(F_\infty,V).\eeq
The following lemma is a standard fact.
\begin{lm}\label{L:33.H}We have
	\[\Proj_F(\bfw^\eta)=\exp_{F,V}\left(\frac{1-p^{-1}\varphi^{-1}}{1-\varphi}\eta\right)\in \rmH^1(F,V).\]
\end{lm}
\begin{proof}Let $g=\eta\ot\bfe$ and let $G(X)\in\bfD(V)\ot \sH_{1,\Qpur}(X)$ such that $(1-\varphi\ot\varphi_\cF)G=g$. Then we have
\[G(\ep_0)=\eta\ot \bfe(\ep_0)-\eta+(1-\varphi)^{-1}\eta.\]
The equation $\psi_\cF\bfe(X)=0$ implies
\[\sum_{\zeta\in \cF^{\Fr^{-1}}[p]}\bfe^{\Fr^{-1}}(X\oplus_\cF \zeta)=0.\]
It follows that
\[\Tr_{F_0/F}(G^{\Fr^{-1}}(\ep_0))=\sum_{\tau\in\Gal(F_0/F)}\eta\ot \bfe(\ep_0^\tau)-\eta+(1-\varphi)^{-1}\eta=\frac{p\varphi-1}{1-\varphi}\eta,\]
and hence
	\begin{align*}
	\Proj_F(\bfw^\eta)&={\rm cor}_{F_0/F}(\Xi_{0,V}(G))=\exp_{F,V}\Tr_{F_0/F}\left(p^{-1}\varphi^{-1}(G^{\Fr^{-1}}(\ep_0))\right)\\
	&=\exp_{F,V}\left((1-p^{-1}\varphi^{-1})(1-\varphi)^{-1} \eta\right).
	\end{align*}
	This completes the proof.
\end{proof}

\begin{lm}\label{L:36.H}Let $\Qp^{\rm cyc}$ be the cyclotomic $\Zp^\x$-extension of $\Q_p$. Let $\sigma_{\rm cyc}\in\Gal(F^\ur_\infty/\Qp)$ be the Frobenius such that $\sigma_{\rm cyc}|_{\Qp^{\rm cyc}}=1$ and $\sigma_{\rm cyc}|_{\Qp^\ur}=\Fr$. For each $\bfz\in \wh\rmH^1(\hat K_\infty,V)$, we have  
	\[\pair{\bfz}{{\rm cor}_{F_\infty/\hat K_\infty}(\bfw^\eta)}_{\hat K_\infty}=\Proj_{\hat K_\infty}(\COL^\eta(\bfz)) \sum_{i=1}^{[F:\Qp]} \frac{\sg_{\rm cyc}^i|_{\hat K_\infty}}{[F_\infty:\hat K_\infty]\cdot h_{\bfe}^{\Fr^i  }}\in \cW\dBr{\hat\Gamma_\infty)}\ot\Qp.\]
\end{lm}
\begin{proof} We first recall that for every $e\in (R\ot_{\cO}\cW)^{\psi_\cF=0}$, the big exponential map $\Exp^\ep_{V,1}(\eta\ot e)$ in \thmref{T:BigExp} is given by \[\Exp^\ep_{V,1}(\eta\ot e)=(\exp_{F_n,V}(\Xi_{n,V}(G_e)))_{n=0,1,2,\ldots},\] where 
$G_e\in \bfD(V)\ot \sH_{1,\Qpur}(X)$ is a solution of $(1-\varphi\ot \varphi_\cF)G_e=\eta\ot e$. By the definition of $G_e$, we verify that
	\beq\label{E:F1.H}\begin{aligned}\Xi_{n,V}(G_e)&=p^{-(n+1)}(\varphi^{-(n+1)}\ot 1)G_e^{\Fr^{-(n+1)}}(\ep_n)\\
		&=\sum_{m=0}^{\infty} (p\varphi)^{-(n+1)}\varphi^m\eta\ot e^{\Fr^{m-(n+1)}}(\ep_{n-m})\\
		&=\sum_{m=0}^{n+1} (p\varphi)^{-(n+1)}\varphi^m\eta\ot e^{\Fr^{m-(n+1)}}(\ep_{n-m})+p^{-(n+1)}(1-\varphi\ot\Fr)^{-1}(\eta\ot e(0)).
	\end{aligned}\eeq
Put $z_n=\Proj_{\hat K_n}(\bfz)$ and $\hat G_n=\Gal(F_n/F)$. Following the computation in \cite[Thm.~5.10]{kobayashi-pr}, we find that $\bpair{\Proj_{\hat K_n}(\COL^{\ep}_e(\bfz))}{\eta}$ is given by 
	\beq\label{E:F2.H}\sum_{m=0}^{\infty} \bpair{\sum_{\gamma\in \hat G_n}\exp^*_{\hat K_n,V}(z_n^{\gamma^{-1}\sg_0^{n+1-m}})\gamma}{\sum_{\tau\in\hat G_n} (p\varphi)^{-(n+1)}\varphi^m\eta\ot e^{\Fr^{m-(n+1)}}(\ep_{n-m})^{\tau\sg_0^{n+1-m}}\tau|_{\hat K_n}}.\eeq
	On the other hand, \[\Proj_{\hat K_n}(\pair{\bfz}{{\rm cor}_{F_\infty/\hat K_\infty}(\bfw^\eta)}_{\hat K_\infty})=\frac{1}{[F_\infty:\hat K_\infty]}\sum_{j=1}^{[F:\Qp]}\Proj_{\hat K_n}(\pair{\bfz^{\sg_0^{-j}}}{\bfw^\eta}_{F_\infty})\sigma_0^j|_{\hat K_\infty},\] and  $\Proj_{\hat K_n}(\pair{\bfz^{\sg_0^{-j}}}{\bfw^\eta}_{F_\infty})$ equals
	\begin{align*}
&\sum_{\gamma\in\hat G_n}\pair{z_n^{\sg_0^{-j}\gamma^{-1}}}{\exp_{F_n,V}(\Xi_{n,V}(G_{\bfe})}_{F_n}\gamma|_{\hat K_\infty}
	=\Tr_{F_n/\Qp}\left(\bpair{\sum_{\gamma\in\hat G_n}\exp^*_{\hat K_n,V}(z_n^{\sg_0^{-j}\gamma^{-1}})\gamma|_{\hat K_\infty}}{\Xi_{n,V}(G_{\bfe})}_V\right)\\
	=&\sum_{m=0}^{\infty}\sum_{i=1}^{[F:\Qp]} \bpair{\sum_{\gamma\in \hat G_n}\exp^*_{\hat K_n,V}(z_n^{\gamma^{-1}\sigma_0^{i-j+n+1-m}})\gamma}{\sum_{\tau\in \hat G_n} (p\varphi)^{-(n+1)}\varphi^m\eta\ot \bfe^{\Fr^{m-(n+1)}}(\ep_{n-m})^{\tau\sg_0^{i+n+1-m}}\tau|_{\hat K_n}}\\
	=&\sum_{i=1}^{[F:\Qp]}\bpair{\Proj_{\hat K_n}(\COL_{\bfe}^{\ep}(\bfz^{\sg_0^{-j}})^{\sg_0^i})}{\eta}.
	\end{align*}
From this, it follows immediately that
\beq\label{E:37.H}\begin{aligned}
\Proj_{\hat K_n}(\pair{\bfz}{{\rm cor}_{F_\infty/\hat K_\infty}(\bfw^\eta)}_{\hat K_\infty})&=\frac{1}{[F_\infty:\hat K_\infty]}\sum_{j=1}^{[F:\Qp]}\sum_{i=1}^{[F:\Qp]}\bpair{\Proj_{\hat K_n}(\COL_{\bfe}^{\ep}(\bfz^{\sg_0^{-j}})^{\sg_0^i})}{\eta}\sigma_0^j\\
&=\frac{1}{[F_\infty:\hat K_\infty]}\sum_{i=1}^{[F:\Qp]}\left(\COL^\eta(\bfz)\right)^{\sg_0^i}\cdot \frac{1}{h_{\bfe}^{\sg_0^i}}.
\end{aligned}\eeq	
On the other hand, by definition, \[\COL^\eta(\bfz)=\sum_{j=1}^{[F:\Qp]}\bpair{\COL^\ep_{g_\rho}(\bfz^{\sg_0^{-j}})}{\eta}\sg_0^j\] with $g_\rho=\rho(1+X)$. From \eqref{E:F2.H} with $e=g_\rho$ and the fact that $g_\rho^{\sg_0^{m-n-1}}(\ep_{n-m})=\zeta_{p^{n+1-m}}\in \Qp^{\rm cyc}$, we deduce that 
	\[\bpair{\COL^\ep_{g_\rho}(\bfz^{\sg_0^{-j}})^{\sg_0^i}}{\eta}= \bpair{\COL^\ep_{g_\rho}(\bfz^{\sg_0^{-j}\sg_{\rm cyc}^i})}{\eta},\]
	so $\left(\COL^\eta(\bfz)\right)^{\sg_0^i}=\COL^\eta(\bfz)\cdot\sg_{\rm cyc}^i$. Now the lemma follows from \eqref{E:37.H}.
\end{proof}


Now we give a formula for the derived $p$-adic heights over $K$ in terms of the Coleman map over $F_\infty$.  
For every prime $v$ of $K$ above $p$, let $\rmH^1_{\rm fin}(K_v,V)\subset\rmH^1(K_v,V)$ be the Bloch--Kato finite subspace, and set
\beq\label{log}
\log_{\omega_E,v}=\pair{\log_{K_v,V}(-)}{\omega_E\ot t^{-1}}_{\rm dR}:
\rmH^1_{\rm fin}(K_v,V)\to L.
\eeq
Since $p$ is a prime of good reduction for $E$, by \cite[Cor.~3.8.4]{BK} we have $\rmH^1_{\rm exp}(K_v,V)=\rmH^1_{\rm fin}(K_v,V)$, 
where $\rmH^1_{\rm exp}(K_v,V)\subset\rmH^1(K_v,V)$ is the image of 
${\rm exp}_{K_v,V}$. For the ease of notation, we write $\COL^\eta(-)$ for $\Proj_{\hat K_\infty}(\COL^\eta(-))$ in what follows.

\begin{prop}\label{P:37.H}Let $z,x\in S_p^{(r)}(E/\base)\ot_{\Qp} L$,  and suppose that there exists $\bfz\in\wh\rmH^1(\base_\infty,V)$ such that $\Proj_K(\bfz)=z$. If $\COL^\eta(\loc_{v}(\bfz))\in J^r\cW\dBr{\hat\Gamma_\infty}\ot\Qp$ for some $v\in\stt{\frakp,\ol{\frakp}}$, then 
	\begin{align*}h^{(r)}(z,x)
	&=-\frac{1-p^{-1}\al_p}{1-\al_p^{-1}}\cdot \frac{1}{[F_\infty:\hat K_\infty]}\\
	\times&\left(\log_{\omega_E,\frakp}(x)\COL^\eta(\loc_\frakp(\bfz))+\log_{\omega_E,\frakp}(\ol{x})\COL^\eta(\loc_{\frakp}(\ol{\bfz}))\right)\pmod{J^{r+1}\cW\dBr{\hat\Gamma_\infty}\ot\Qp},\end{align*}
where $\ol{x}$ and $\ol{\bfz}$ are the complex conjugates of $x$ and $\bfz$.
\end{prop}
\begin{proof}Let $\bfw_p:={\rm cor}_{F_\infty/\hat K_\infty}(\bfw^\eta)\in \wh\rmH^1_{\rm fin}(\hat K_\infty,V)$. Since $\dim_{\Qp}\rmH^1_{\rm fin}(\Qp,V)=1$, we can write
	\[\loc_\frakp(x)=c\cdot \Proj_{\Qp}(\bfw_p)=c\cdot {\rm cor}_{F/\Qp}(\Proj_F(\bfw^\eta))\] for some $c\in\Qp$. By \lmref{L:33.H},
	\[\pair{\log_{\Q_p,V}(\loc_\frakp(x))}{\omega_E\ot t^{-1}}_{\dR}=c[F:\Qp]\cdot \left\langle\frac{1-p^{-1}\varphi^{-1}}{1-\varphi}\eta,\omega_E\ot t^{-1}\right\rangle_{\dR}.\]
	Since $\varphi\eta=p^{-1}\alpha_p\cdot\eta$, this shows that
	\[
	c=\frac{1-p^{-1}\al_p}{1-\al_p^{-1}}\cdot [F:\Qp]^{-1}\cdot\log_{\omega_E,\frakp}(x).
	\] 
Applying \corref{C:13.H}, we find that that 
\begin{align*}&h^{(r)}(z,x)=-(1-p^{-1}\al_p)(1-\al_p^{-1})^{-1}[F:\Qp]^{-1}\\
	&\times\left(\log_{\omega_E,\frakp}(x)\cdot \pair{\loc_\frakp(\bfz)}{\bfw_p}_{\hat K_\infty}+
	\log_{\omega_E,\frakp}(\ol{x})\cdot \pair{\loc_\frakp(\ol{\bfz})}{\bfw_p}_{\hat K_\infty}
	\right)\pmod{J^{r+1}}.\end{align*}
	Since $\rho(1+X)=h_{\bfe}\cdot \bfe $ and $\bfe(0)=1$, we find that 
	$1=\bfe(0)\cdot (h_{\bfe}|_{\gen=1})$ and hence $h_{\bfe}\con 1\pmod{J}$. The assertion now follows from the above equation and \lmref{L:36.H}.\end{proof}



\section{Euler system construction of theta elements}
\label{sec:GKC}
\def\fraka{\mathfrak a}
In this section we prove Theorem~\ref{thm:theta-Euler}, recovering the square-root anticyclotomic $p$-adic $L$-functions of Bertolini--Darmon \cite{BDmumford-tate} (in the definite case) 
as the image of a $p$-adic family of diagonal cycles 
\cite{DR2} under 
the Coleman map of $\S\ref{subsec:Col}$.

\subsection{Ordinary $\Lambda$-adic forms}\label{subsec:Hida}

Fix a prime $p>2$. Let $\cR$ be a normal domain finite flat over $\Lambda:=\cO\dBr{1+p\bZ_p}$, where $\cO$ is the ring of integers of a finite extension $L/\bQ_p$. We say that a point $x\in{\rm Spec}\;\cR(\overline{\bQ}_p)$ is \emph{locally algebraic} if its restriction to $1+p\bZ_p$ is given by $x(\gamma)=\gamma^{k_x}\epsilon_x(\gamma)$ for some integer $k_x$, called the \emph{weight} of $x$, and some finite order character $\epsilon_x:1+p\bZ_p\rightarrow\mu_{p^\infty}$; we say that $x$ is \emph{arithmetic} if it has weight $k_x\geqslant 2$. Let $\mathfrak{X}_\cR^+$ be the set of arithmetic points. 

Fix a positive integer $N$ prime to $p$, and let $\chi:(\bZ/Np\bZ)^\times\rightarrow\cO^\times$ be a Dirichlet character modulo $Np$. Let $S^{o}(N,\chi,\cR)$ be the space of \emph{ordinary $\cR$-adic cusp forms} of tame level $N$ and branch character $\chi$, consisting of formal power series 
\[
\boldsymbol{f}(q)=\sum_{n=1}^\infty a_n(\boldsymbol{f})q^n\in\cR\dBr{q}
\] 
such that for every $x\in\mathfrak{X}_\cR^+$ the specialization $\boldsymbol{f}_x(q)$ 
is the $q$-expansion of a $p$-ordinary cusp form $\boldsymbol{f}_x\in S_{k_x}(Np^{r_x+1},\chi\omega^{2-k_x}\epsilon_x)$. Here $r_x\geqslant 0$ 
is such that $\epsilon_x(1+p)$ has exact order $p^{r_x}$, and  $\omega:(\bZ/p\bZ)^\times\rightarrow\mu_{p-1}$ is the Teichm\"uller character. 

We say that $\boldsymbol{f}\in S^o(N,\chi,\cR)$ is a \emph{primitive Hida family}    
if for every $x\in\mathfrak{X}_\cR^+$ we have that $\boldsymbol{f}_x$ is an ordinary $p$-stabilized newform (in the sense of \cite[Def.~2.4]{hsieh-triple}) of tame level $N$.  
Given a primitive Hida family  $\bff\in S^o(N,\chi,\cR)$, and writing $\chi=\chi'\chi_p$ with $\chi'$ (resp. $\chi_p$) a Dirichlet modulo $N$ (resp. $p$), there is a primitive ${\bff}^\iota\in S^o(N,\chi_p\overline{\chi}',\cR)$ with Fourier coefficients
\[
a_\ell({\bff}^\iota)=\left\{
\begin{array}{ll}
\overline{\chi}'(\ell)a_\ell(\bff)&\textrm{if $\ell\nmid N$,}\\
a_\ell(\bff)^{-1}\chi_p\omega^2(\ell)\langle\ell\rangle_\cR\ell^{-1}&
\textrm{if $\ell\mid N$,}
\end{array}
\right.
\]
having the property that for every $x\in\mathfrak{X}_\cR^+$ the specialization ${\bff}^\iota_x$ is the $p$-stabilized newform attached to the character twist $\bff_x\otimes\overline{\chi}'$.

By \cite{hida86b} (\emph{cf.} \cite[Thm.~2.2.1]{wiles88}), attached to every primitive Hida family $\bff\in S^o(N,\chi,\cR)$ there is a continuous $\cR$-adic representation $\rho_{\boldsymbol{f}}:G_\bQ\rightarrow{\rm GL}_2({\rm Frac}\;\cR)$ 
which is unramified outside $Np$ and such that for every prime $\ell\nmid Np$,
\[
{\rm tr}\;\rho_{\boldsymbol{f}}({\rm Frob}_\ell)=a_\ell(\bff),
\quad
{\rm det}\;\rho_{\bff}({\rm Frob}_\ell)=\chi\omega^2(\ell)\langle\ell\rangle_\cR\ell^{-1},
\]
where 
$\langle\ell\rangle\in\cR^\times$ is the image of $\ell\omega^{-1}(\ell)$ under the natural map $1+p\bZ_p\rightarrow\cO\dBr{1+p\bZ_p}^\times=\Lambda^\times\rightarrow\cR^\times$. In particular, letting 
$\langle\varepsilon_{\rm cyc}\rangle_\cR:G_\bQ\rightarrow\cR^\times$ be defined by $\langle\varepsilon_{\rm cyc}\rangle_{\cR}(\sigma)=\langle\varepsilon_{\rm cyc}(\sigma)\rangle_{\cR}$, it follows that $\rho_{\boldsymbol{f}}$ has determinant $\chi_\cR^{-1}\varepsilon_{\rm cyc}^{-1}$, where $\chi_\cR:G_\bQ\rightarrow\cR^\times$ is given by $\chi_\cR:=\sigma_\chi\langle\varepsilon_{\rm cyc}\rangle^{-2}\langle\varepsilon_{\rm cyc}\rangle_{\cR}$, with $\sigma_\chi$ the Galois character sending ${\rm Frob}_\ell\mapsto\chi(\ell)^{-1}$. Moreover, by \cite[Thm.~2.2.2]{wiles88} the restriction of $\rho_{\boldsymbol{f}}$ to $G_{\bQ_p}$ is given by
\begin{equation}\label{eq:wiles}
\rho_{\bff}\vert_{G_{\bQ_p}}\sim\left(\begin{array}{cc}\psi_\bff&*\\
0&\psi_{\bff}^{-1}\chi_{\cR}^{-1}\varepsilon_{\rm cyc}^{-1}\end{array}\right)
\end{equation}
where $\psi_{\bff}:G_{\bQ_p}\rightarrow\cR^\times$ is the unramified character with $\psi_\bff({\rm Frob}_p)=a_p(\bff)$. 




\subsection{Triple product $p$-adic $L$-function}\label{subsec:tripleL}

Let 
\[
(\bff,\bfg,\bfh)\in S^o(N_\bff,\chi_{\bff},\cR_\bff)\times S^o(N_\bfg,\chi_{\bfg},\cR_\bfg)\times S^o(N_\bfh,\chi_{\bfh},\cR_\bfh)
\]
be a triple of primitive Hida families. 
%
Set 
\[
\mathcal{R}:=\cR_{\bff}\hat{\otimes}_{\cO}\cR_{\bfg}\hat{\otimes}_{\cO}\cR_{\bfh}, 
\]
which is a finite extension of the three-variable Iwasawa algebra $\mathcal{R}_0:=\Lambda\hat{\otimes}_\cO\Lambda\hat{\otimes}_\cO\Lambda$, 
and define the weight space $\mathfrak{X}_{\mathcal{R}}^\bff$ for the triple $(\bff,\bfg,\bfh)$ in the \emph{$\bff$-dominated unbalanced range} by
\begin{equation}\label{eq:bal}
\mathfrak{X}_{\mathcal{R}}^\bff:=\left\{
(x,y,z)\in\mathfrak{X}_{\cR_\bff}^+\times\mathfrak{X}_{\cR_\bfg}^{\rm cls}\times\mathfrak{X}_{\cR_\bfh}^{\rm cls}\;\colon\;k_x\geqslant k_y+k_z\;\textrm{and}\;k_x\equiv k_y+k_z\;({\rm mod}\;2)\right\},
\end{equation}
where $\mathfrak{X}_{\cR_{\bfg}}^{\rm cls}\supset\mathfrak{X}_{\cR_{\bfg}}^+$ (and similarly $\mathfrak{X}_{\cR_{\bfh}}^{\rm cls}$) is the set of locally algebraic points in ${\rm Spec}\;{\cR_\bfg}(\overline{\bQ}_p)$ for which $\bfg_x(q)$ is the $q$-expansion of a classical modular form.

For $\boldsymbol{\phi}\in\{\bff,\bfg,\bfh\}$ and a positive integer $N$ prime to $p$ and divisible by $N_{\boldsymbol{\phi}}$, define the space of \emph{$\Lambda$-adic test vectors} $S^o(N,\chi_{\boldsymbol{\phi}},\cR_{\boldsymbol{\phi}})[\boldsymbol{\phi}]$ to be the $\cR_{\boldsymbol{\phi}}$-submodule of $S^o(N,\chi_{\boldsymbol{\phi}},\cR_{\boldsymbol{\phi}})$ generated by $\{\boldsymbol{\phi}(q^d)\}$, as $d$ ranges over the positive divisors of $N/N_{\boldsymbol{\phi}}$.

For the next result, set $N:={\rm lcm}(N_\bff,N_\bfg,N_\bfh)$, and consider the following hypothesis: 
\begin{equation}\label{eq:hyp-sigma}
\textrm{for some $(x,y,z)\in\mathfrak{X}_{\mathcal{R}}^\bff$, we have  $\eps_q(\bff_x^\circ,\bfg_y^\circ,\bfh_z^\circ)=+1$ for all $q\mid N$.}\tag{$\Sigma^-$}
\end{equation}
Here $\eps_q(\bff_x^\circ,\bfg_y^\circ,\bfh_z^\circ)$ denotes the local root number of the Kummer self-dual twist of the Galois representations attached to the newforms $\bff_x^\circ$, $\bfg_y^\circ$, and $\bfh_z^\circ$ corresponding to $\bff_x$, $\bfg_y$, and $\bfh_z$, respectively. 

\begin{thm}\label{thm:triple-L}
Assume that the residual representation $\bar{\rho}_{\bff}$ satisfies
\begin{itemize}
\item[{\rm (CR)}] $\bar\rho_{\bff}$ is absolutely irreducible and $p$-distinguished,
\end{itemize}
and that, in addition to $(\Sigma^-)$, the triple $(\bff,\bfg,\bfh)$ satisfies
\begin{itemize}
\item[{\rm (ev)}] $\chi_{\bff}\chi_{\bfg}\chi_{\bfh}=\omega^{2a}$ for some $a\in\bZ$,
\item[{\rm (sq)}] ${\rm gcd}(N_\bff,N_\bfg,N_\bfh)$ is square-free.
\end{itemize}
Then there exist $\Lambda$-adic test vectors $(\breve{\bff}^\star,\breve{\bfg}^\star,\breve{\bfh}^\star)$ and an element 
\[
\mathscr{L}_p^f(\breve{\bff}^\star,\breve{\bfg}^\star,\breve{\bfh}^\star)\in\mathcal{R}
\]
such that for all $(x,y,z)\in\mathfrak{X}_{\mathcal{R}}^\bff$ of weight $(k,\ell,m)$:
\[
\nu_{(x,y,z)}(\mathscr{L}_p^f(\breve{\bff}^\star,\breve{\bfg}^\star,\breve{\bfh}^\star)^2)=\frac{\Gamma(k,\ell,m)}{2^{\alpha(k,\ell,m)}}\cdot\frac{\mathcal{E}(\bff_x,\bfg_y,\bfh_z)^2}{\mathcal{E}_0(\bff_x)^2\cdot\mathcal{E}_1(\bff_x)^2}\cdot\prod_{q\mid N}c_q\cdot\frac{L(\bff_x^\circ\ot \bfg_y^\circ\ot \bfh_z^\circ,c)}{\pi^{2(k-2)}\cdot\lVert\bff_x^\circ\rVert^2},
\]
where: 
\begin{itemize}
    \item $c=(k+\ell+m-2)/2$,
	\item $\Gamma(k,\ell,m)=(c-1)!\cdot(c-m)!\cdot(c-\ell)!\cdot(c+1-\ell-m)!$,
	\item $\alpha(k,\ell,m)\in\mathcal{R}$ is a linear form in the variables $k$, $\ell$, $m$,
	\item $\mathcal{E}(\bff_x,\bfg_y,\bfh_z)=(1-\frac{\beta_{\bff_x}\alpha_{\bfg_y}\alpha_{\bfh_z}}{p^c})(1-\frac{\beta_{\bff_x}\beta_{\bfg_y}\alpha_{\bfh_z}}{p^c})(1-\frac{\beta_{\bff_x}\alpha_{\bfg_y}\beta_{\bfh_z}}{p^c})(1-\frac{\beta_{\bff_x}\beta_{\bfg_y}\beta_{\bfh_z}}{p^c})$,
\item
$\mathcal{E}_0(\bff_x)=(1-\frac{\beta_{\bff_x}}{\alpha_{\bff_x}})$, 
$\mathcal{E}_1(\bff_x)=(1-\frac{\beta_{\bff_x}}{p\alpha_{\bff_x}})$,
\end{itemize}
and $\lVert\bff_x^\circ\rVert^2$ is the Petersson norm of $\bff^\circ_x$ on $\Gamma_0(N_\bff)$.
\end{thm}

\begin{proof}
See \cite[Thm.~A]{hsieh-triple}. More specifically, the construction of $\mathscr{L}_p^f(\breve{\bff}^\star,\breve{\bfg}^\star,\breve{\bfh}^\star)$ under hypotheses (CR), (ev), and (sq) is given in \cite[\S{3.6}]{hsieh-triple} (where it is denoted $\mathscr{L}_{\boldsymbol{F}}^\bff$), and the proof of its interpolation property assuming ($\Sigma^-$) is contained in \cite[\S{7}]{hsieh-triple}. 
\end{proof}

\subsection{Triple tensor product of big Galois representations}
\label{subsec:Gal}

Let $(\bff,\bfg,\bfh)$ be a triple of primitive Hida families with $\chi_{\bff}\chi_{\bfg}\chi_{\bfh}=\omega^{2a}$ for some $a\in\bZ$. 
For $\boldsymbol{\phi}\in\{\bff,\bfg,\bfh\}$, let $V_{\boldsymbol{\phi}}$ be the natural lattice in $({\rm Frac}\;\cR_{\boldsymbol{\phi}})^2$ realizing the Galois representation $\rho_{\boldsymbol{\phi}}$ in the \'etale cohomology of modular curves (see \cite{OhtaII}), and set
\[
\bV_{\bff\bfg\bfh}:=V_\bff\hat\otimes_{\cO}V_{\bfg}\hat\otimes_{\cO}V_{\bfh}.
\]
This has rank $8$ over $\mathcal{R}$, and by hypothesis its determinant can be written as $\det\bV_{\bff\bfg\bfh}=\mathcal{X}^2\varepsilon_{\rm cyc}$ for a $p$-ramified Galois character $\mathcal{X}$ taking the value $(-1)^a$ at complex conjugation. 
Similarly as in \cite[Def.~2.1.3]{howard-invmath}, 
we define the \emph{critical twist} 
\begin{equation}\label{eq:crit}
\bV_{\bff\bfg\bfh}^\dagger:=\bV_{\bff\bfg\bfh}\otimes\mathcal{X}^{-1}.\nonumber
\end{equation}
More generally, for any multiple $N$ of $N_{\boldsymbol{\phi}}$ one can define Galois modules $V_{\boldsymbol{\phi}}(N)$ by working in tame level $N$; these split non-canonically into a finite direct sum of the $\cR_{\boldsymbol{\phi}}$-adic representations $V_{\boldsymbol{\phi}}$ (see \cite[\S{1.5.3}]{DR2}), 
and they define $\mathbb{V}_{\bff\bfg\bfh}^\dagger(N)$ for any $N$ divisible by ${\rm lcm}(N_{\bff},N_{\bfg},N_{\bfh})$. 

If $f$ is a classical specialization of $\bff$ with associated $p$-adic Galois representation $V_f$, we let $\mathbb{V}_{f,\bfg\bfh}$ be the quotient of $\mathbb{V}_{\bff\bfg\bfh}$ given by
\[
\bV_{f,{\bfg\bfh}}:=V_f\otimes_{\cO}V_{\bfg}\hat{\otimes}_{\cR}V_{\bfh}.
\]
Denote by $\bV^\dagger_{f,{\bfg\bfh}}$ the corresponding quotient  of $\bV^\dagger_{\bff\bfg\bfh}$, and by $\bV^\dagger_{f,{\bfg\bfh}}(N)$ its level $N$ counterpart.

\subsection{Theta elements and factorization} 
\label{subsec:theta}

We recall the factorization proven in \cite[\S{8}]{hsieh-triple}. Let $f\in S_2(pN_f)$ be a $p$-stabilized newform of tame level $N_f$ defined over $\cO$, let $f^\circ\in S_2(N_f)$ be the associated newform, and let $\alpha_p=\alpha_p(f)\in\cO^\times$ be the $U_p$-eigenvalue of $f$. Let $K$ be an imaginary quadratic field of discriminant $D_K$ prime to $N_f$. Write 
\[
N_f=N^+ N^-
\] 
with $N^+$ (resp. $N^-$) divisible only by primes which are split (resp. inert) in $K$, and choose an ideal $\mathfrak{N}^+\subset\cO_K$ with $\cO_K/\mathfrak{N}^+\iso\bZ/N^+\bZ$. 

We assume that $p\cO_K=\pp\overline{\pp}$ splits in $K$, with $\pp$ the prime of $K$ above $p$ induced by our fixed embedding $\iota_p:\overline{\bQ}\hookrightarrow\bC_p$. Let $\Gamma_\infty={\rm Gal}(K_\infty/K)$ be the Galois group of the anticyclotomic $\bZ_p$-extension of $K$, fix a topological generator $\gen\in\Gamma_\infty$, and identity $\cO\dBr{\Gamma_\infty}$ with the one-variable power series ring $\cO\dBr{T}$ via $\gen\mapsto 1+T$. For any prime-to-$p$ ideal $\mathfrak a$ of $K$, let $\sigma_{\mathfrak a}$ be the image of $\mathfrak a$ in the Galois group of the ray class field $K(p^\infty)/K$ of conductor $p^\infty$ under the geometrically normalized reciprocity law map.

\begin{thm}\label{thm:BD-theta}
Let $\chi$ be a ring class character of $K$ of conductor $c\cO_K$ with values in $\cO$, and
assume that: 
\begin{itemize}
	\item[(i)] {} $(pN_f,cD_K)=1$,
    \item[(ii)] {} $N^-$ is the square-free product of an odd number of primes,
    \item[(iii)] {} $\bar{\rho}_f$ is absolutely irreducible and $p$-distinguished, 
    \item[(iv)] if $q\vert N^-$ is a prime with $q\equiv 1\pmod{p}$, then $\bar{\rho}_f$ is ramified at $q$.
\end{itemize}
Then there exists a unique element $\Theta_{f/K,\chi}(T)\in\cO\pwseries{T}$ such that for every $p$-power root of unity $\zeta\in\overline{\bQ}_p$:
\[
\Theta^{}_{f/K,\chi}(\zeta-1)^2=\frac{p^n}{\alpha_p^{2n}}\cdot\mathcal{E}_p(f,\chi,\zeta)^{2}\cdot\frac{L(f^\circ/K\otimes\chi\epsilon_\zeta,1)}{(2\pi)^2\cdot\Omega_{f^\circ,N^-}}\cdot u_K^2\sqrt{D_K}\chi\epsilon_\zeta(\sigma_{\mathfrak{N}^+})\cdot\eps_p,
\]	
where:
\begin{itemize}
\item $n\geqslant 0$ is such that $\zeta$ has exact order $p^n$, 
\item $\epsilon_\zeta:\Gamma_\infty\rightarrow\mu_{p^\infty}$ be the character defined by $\epsilon_\zeta(\gen)=\zeta$,
\item 
$\mathcal{E}_p(f,\chi,\zeta)=
\left\{
\begin{array}{ll}
(1-\alpha_p^{-1}\chi(\pp))(1-\alpha_p\chi(\overline{\pp}))&
\textrm{if $n=0$,}\\ [0.5 em]
1&\textrm{if $n>0$,}
\end{array}
\right.$
\item 
$\Omega_{f^\circ,N^-}=4\cdot\lVert f^\circ\rVert_{\Gamma_0(N_{f^\circ})}^2\cdot\eta_{f^\circ,N^-}^{-1}$ is the Gross period of $f^\circ$,
\item $\sigma_{\mathfrak{N}^+}\in\Gamma_\infty$ is the image of $\mathfrak{N}^+$ under the geometrically normalized Artin's reciprocity map, 
\item $u_K=\vert\cO_K^\times\vert/2$, and
$\eps_p\in\{\pm 1\}$ is the local root number of $f^\circ$ at $p$. 
\end{itemize}
\end{thm}

\begin{proof}
See \cite{BDmumford-tate} for the first construction, and  
\cite[Thm.~A]{ChHs1} for the stated interpolation property. 
\end{proof}

When $\chi$ is the trivial character, we write $\Theta_{f/K,\chi}(T)$ simply as $\Theta_{f/K}(T)$. Suppose now that $f$  
is the specialization of a primitive Hida family $\bff\in S^o(N_f,\cR)$ with branch character $\chi_{\bff}=\mathds{1}$ at an arithmetic point $x_1\in\mathfrak{X}_\cR^+$ of weight $2$.  
Let $\ell\nmid pN_f$ be a prime split in $K$, and let $\chi$ be a ring class character of $K$ of conductor $\ell^m\cO_K$ for some even $m>0$. Set $C=D_K\ell^{2m}$ and let
\[
\bfg=\boldsymbol{\theta}_{\chi}(S_2)\in S^o(C,\omega^{-1}\eta_{K/\bQ},\cO\dBr{S_2}),\quad
\bfh=\boldsymbol{\theta}_{\chi^{-1}}(S_3)\in S^o(C,\omega^{-1}\eta_{K/\bQ},\cO\dBr{S_3})
\]  
be the primitive CM Hida families constructed in \cite[\S{8.3}]{hsieh-triple}, where $\eta_{K/\bQ}$ is the quadratic character associated to $K$.  
The $p$-adic triple product $L$-function of Theorem~\ref{thm:triple-L} for this triple $(\bff,\bfg,\bfh)$ is an element in $\mathcal{R}=\cR\dBr{S_2,S_3}$; in the following we let 
\[
\mathscr{L}_p^f(\breve{f}^\star,{\breve{\bfg}^\star\breve{\bfh}^\star})\in\cO\dBr{S}
\] 
denote the restriction to the ``line'' $S=S_2=S_3$ of its image under the specialization map at $x_1$. 

Let $\mathbb{K}_\infty$ 
be the $\Z_p^2$-extension of $K$, and let $K_{\pp^\infty}$ denote the $\pp$-ramified $\Z_p$-extension in $\mathbb{K}_\infty$, with Galois group $\Gamma_{\pp^\infty}={\rm Gal}(K_{\pp^\infty}/K)$. Let $\gamma_\pp\in\Gamma_{\pp^\infty}$ be a topological generator, and for the formal  variable $T$ let $\Psi_T:{\rm Gal}(\mathbb{K}_\infty/K)\rightarrow\cO\dBr{T}^\times$ be the universal character defined by
\begin{equation}\label{eq:univ}
\Psi_T(\sigma)=(1+T)^{l(\sigma)},\quad\textrm{where 
	$\sigma\vert_{K_{\pp^\infty}}=\gamma_\pp^{l(\sigma)}$}.
\end{equation}
Denoting by the superscript $c$ the action of the non-trivial automorphism of $K/\bQ$, the character $\Psi_T^{1-c}$ factors through $\Gamma_\infty$ and yields an identification $\cO\dBr{\Gamma_\infty}\iso\cO\dBr{T}$ corresponding to the topological generator $\gamma_\pp^{1-c}\in\Gamma_\infty$. Let $p^b$ be the order of the $p$-part of the class number of $K$. Hereafter, we shall fix $\mathbf v\in\Zbar_p^\times$ such that $\mathbf v^{p^b}=\varepsilon_{\rm cyc}(\gamma_\frakp^{p^b})\in 1+p\Zp$. 
Let $K(\chi,\al_p)/K$ (resp. $K(\chi)/K$) be the finite extension obtained by adjoining to $K$ the values of $\chi$ and $\al_p$ (resp. the values of $\chi$). 

\begin{prop}\label{prop:factor}
Set $T=\mathbf{v}^{-1}(1+S)-1$. Then 
\[
\mathscr{L}_p^f(\breve{f}^\star,{\breve{\bfg}^\star\breve{\bfh}^\star})=
\pm \Psi_T^{c-1}(\sigma_{\mathfrak{N}^+})\cdot\Theta_{f/K}(T)\cdot C_{f,\chi}\cdot \sqrt{L^{\rm alg}(f/K\ot\chi^2,1)},
\]
where $C_{f,\chi}\in K(\chi,\al_p)^\x$and  \[L^{\rm alg}(f/K\ot\chi^2,1):=\frac{L(f/K\otimes\chi^2,1)}{\pi^2\Omega_{f^\circ,N^-}}\in K(\chi).\] 
\end{prop}

\begin{proof} 	
This is the factorization formula of \cite[Prop.~8.1]{hsieh-triple} specialized to $S=S_2=S_3$, using the interpolation property 
of $\Theta_{f/K,\chi^2}(T)$ at $\zeta=1$.	
\end{proof}

\begin{rem}
The factorization of Proposition~\ref{prop:factor} reflects the decomposition of Galois representations
\begin{equation}\label{eq:dec}
\mathbb{V}_{f,\bfg\bfh}^\dagger=\bigl(V_f(1)\ot{\rm Ind}_K^\Q\Psi_T^{1-c}\bigr)\oplus
\bigl(V_f(1)\ot{\rm Ind}_K^\Q\chi^2\bigr).
\end{equation}
\end{rem}



\subsection{Euler system construction of theta elements}



For the rest of the paper, assume that $f$, $\bfg=\boldsymbol{\theta}_{\chi}(S)$, and $\bfh=\boldsymbol{\theta}_{\chi^{-1}}(S)$ are as in $\S\ref{subsec:theta}$, viewing the latter two in $S^o(C,\omega^{-1}\eta_{K/\bQ},\cO\dBr{S})$. Keeping the notations from $\S\ref{subsec:Gal}$, by \cite[\S{1}]{DR2.5} 
there exists a class
\begin{equation}\label{eq:1.15}
\kappa(f,{\bfg\bfh})\in\rmH^1(\bQ,\bV^\da_{f,{\bfg\bfh}}(N))
\end{equation}
constructed from twisted diagonal cycles on the triple product of modular curves of tame level $N$ (we shall briefly recall the construction of this class in Theorem~\ref{thm:DR-erl} below), where we may take $N={\rm lcm}(N_f,C)$. 

Every triple of test vectors $\breve{\boldsymbol{F}}=(\breve{f},\breve{\bfg},\breve{\bfh})$ defines a Galois-equivariant projection 
\[
{\rm pr}_{\breve{\boldsymbol{F}}}:\rmH^1(\Q,\bV^\dagger_{f,{\bfg\bfh}}(N))\to \rmH^1(\Q,\bV^\dagger_{f,{\bfg\bfh}})
\] 
and we let 
\begin{equation}\label{eq:prF}
\kappa(\breve{f},{\breve{\bfg}\breve{\bfh}}):={\rm pr}_{\breve{\boldsymbol{F}}}(\kappa(f,{\bfg\bfh}))\in\rmH^1(\Q,\bV^\dagger_{f,{\bfg\bfh}}).
\end{equation} 
Since $\Psi_T^{1-c}$ gives the universal character of $\Gamma_\infty={\rm Gal}(K_\infty/K)$, by $(\ref{eq:dec})$ and Shapiro's lemma we have the equalities 
\begin{equation}\label{eq:dec-cohom}
\begin{split}
\rmH^1(\Q,\bV_{f,\bfg\bfh}^\dagger)&=\rmH^1(\Q,V_f(1)\ot{\rm Ind}_K^\Q\Psi_T^{1-c})\oplus\rmH^1(\Q,V_f(1)\ot\Ind_K^\Q\chi^2)\\
&=\widehat{\rmH}^1(K_\infty,V_f(1))\oplus\rmH^1(K,V_f(1)\ot\chi^2).
\end{split}
\end{equation}
Let $g$ and $h$ be the weight $1$ eigenform $\theta_\chi$ and $\theta_{\chi^{-1}}$, respectively, so that the specialization of $(\bfg,\bfh)$ at $T=0$ ($\Leftrightarrow S=\mathbf v-1$) is a $p$-stabilization of the pair $(g,h)$.

\begin{lem}\label{lem:ERL}
Assume that $L(f\ot g\ot h,1)=0$ and that $L(f/K\ot\chi^2,1)\neq 0$. Then for every choice of test vectors $\breve{\boldsymbol{F}}=(\breve{f},\breve{\bfg},\breve{\bfh})$ we have:
\begin{enumerate}
 \item $\kappa(\breve{f},\breve{\bfg}\breve{\bfh})\in \widehat{\rmH}^1(K_\infty,V_f(1))$.
 \item $\loc_{\overline\pp}(\kappa(\breve{f},\breve{\bfg}\breve{\bfh}))=0\in\widehat{\rmH}^1(K_{\infty,\overline\pp},V_f(1))$.
 \end{enumerate}
\end{lem}

\begin{proof} Let $\bfkappa=\kappa(\breve{f},\breve{\bfg}\breve{\bfh})$ and for every $?\in \stt{f,\bfg,\bfh}$, let $F^0V_?$ be the rank one subspace of $V_?$ fixed by the inertia group at $p$. By (\ref{eq:dec-cohom}), in order to prove $(1)$ it suffices to show that some specialization of $\bf\kappa$ has trivial image in $\rmH^1(K,V_f(1)\ot\chi^2)$. Let 
\[
\kappa_{\breve{f},\breve{g}\breve{h}}:=\bfkappa\vert_{S=\mathbf{v}-1}\in\rmH^1(\Q,V_{fgh})=\rmH^1(K,V_f(1))\oplus\rmH^1(K,V_f(1)\ot\chi^2),
\]	
where $V_{fgh}:=V_f(1)\ot V_g\ot V_h$. As noted in \cite[p.~634]{DR2}, the Selmer group ${\rm Sel}(\bQ,V_{fgh})\subset\rmH^1(\Q,V_{fgh})$ is given by 
\[
{\rm Sel}(\bQ,V_{fgh})={\rm ker}\biggl( \rmH^1(\bQ,V_{fgh})\overset{\partial_p\circ{\rm loc}_p}\rightarrow \rmH^1(\bQ_p,V_f^-(1)\otimes V_{g}\ot V_h)\biggr),
\]
where $\partial_p$ 
is the natural map induced by the projection $V_f\twoheadrightarrow V_f^-:=V_f/F^0V_f$, and so
\begin{equation}\label{eq:Sel-dec}
{\rm Sel}(\bQ,V_{fgh})={\rm Sel}(K,V_f(1))\oplus{\rm Sel}(K,V_f(1)\ot\chi^2).
\end{equation}
The implications 
$L(f\ot g\ot h,1)=0\Rightarrow\kappa_{\breve{f},\breve{g}\breve{h}}\in{\rm Sel}(\bQ,V_{fgh})$ and $L(f/K\ot\chi^2,1)\neq 0\Rightarrow{\rm Sel}(K,V_f(1)\ot\chi^2)=0$, which follow from \cite[Thm.~C]{DR2} and 
\cite[Thm.~1]{ChHs2}, respectively, thus yield assertion (1).

We proceed to prove (2). We know that the local class $\loc_p(\bfkappa)$ belongs to $\rmH^1(\Qp,F^+\mathbb{V}_{f\bfg\bfh}^\dagger)$, where
\begin{align*}
&F^+\mathbb{V}_{f\bfg\bfh}^\dagger:=\left(F^0V_f(1)\ot F^0V_{\bfg}\ot V_{\bfh}+F^0V_f(1)\ot V_{\bfg}\ot F^0V_{\bfh}+V_f(1)\ot F^0V_{\bfg}\ot F^0V_{\bfh}\right)\ot\cX^{-1}\end{align*}
is a rank four subspace of $\mathbb{V}_{f\bfg\bfh}^\dagger$ (see \cite[Cor.~2.3]{DR2}). In our case where $(\bfg,\bfh)=(\boldsymbol{\theta}_{\chi},\boldsymbol{\theta}_{\chi^{-1}})$, we have
\[F^+\mathbb{V}_{f\bfg\bfh}^\dagger=V_f(1)\ot \Psi_T^{1-c}+F^0V_f(1)\ot(\chi^{2}\oplus\chi^{-2}),\]
where $\Psi_T$ is viewed as a character of $G_{\Qp}$ via the embedding $K\hookto \Qp$ induced by $\frakp$. From part (1) of the lemma, it follows that 
\begin{align*}\loc_p(\bfkappa)&=(\loc_{\frakp}(\bfkappa),\loc_{\frakp}(\overline{\bfkappa}))\in \rmH^1(K_\frakp,V_f(1)\ot \Psi_T^{1-c})\oplus\stt{0}\\
&\subset \rmH^1(K_\frakp,V_f(1)\ot \Psi_T^{1-c})\oplus  \rmH^1(K_\frakp,V_f(1)\ot\Psi_T^{c-1})=\rmH^1(\Qp,V_f(1)\ot\Ind_K^\Q\Psi_T^{1-c}).
\end{align*}
We thus conclude that $\loc_{\frakp}(\ol{\bfkappa})=0$, and hence $\loc_{\ol{\frakp}}(\bfkappa)=0$. 
\end{proof}

From now on, assume that $f^\circ\in S_2(N_f)$ is the newform corresponding to an elliptic curve $E/\bQ$ with good ordinary reduction at $p$. In particular, $V_f(1)\iso V_pE$, and under the conditions in Lemma~\ref{lem:ERL}
we have the class $\kappa(\breve{f},\breve{\bfg}\breve{\bfh})\in\widehat{\rmH}^1(K_\infty,V_pE\ot L)$. 


The following key theorem is a variant of the ``explicit reciprocity law'' of \cite[Thm.~5.3]{DR2} in our setting in terms of the Coleman map constructed in $\ref{subsec:Col}$.

\begin{thm}[Darmon--Rotger]
\label{thm:DR-erl}
%
Assume that $L(f\ot g\ot h,1)=0$ and that $L(f/K\ot\chi^2,1)\neq 0$.  
Then ${\rm loc}_{\overline\pp}(\kappa(\breve{f}^\star,\breve{\bfg}^\star\breve{\bfh}^\star))=0$ and
\begin{equation}\label{eq:ERL}
\mathscr{L}_p^f(\breve{f}^\star,{\breve{\bfg}^\star\breve{\bfh}^\star})
=\alpha_p/2\cdot(1-\alpha_p^{-1}a_p(\bfg)a_p(\bfh)^{-1})\cdot
\COL^\eta({\rm loc}_\pp(\kappa(\breve{f}^\star,\breve{\bfg}^\star\breve{\bfh}^\star))),
\end{equation}
where $\breve{\boldsymbol{F}}^\star=(\breve{f}^\star,\breve{\bfg}^\star,\breve{\bfh}^\star)$ is the triple of test vectors from Theorem~\ref{thm:triple-L}. 
\end{thm}

\begin{proof}
The first claim is immediate from Lemma~\ref{lem:ERL}. For the proof of the second, we begin by briefly recalling from \cite[\S{1}]{DR2} 
the construction of the class $\kappa(f,{\bfg\bfh})$ in $(\ref{eq:1.15})$. In the following, all references are to \cite{DR2} unless otherwise stated.

Consider the triple product of modular curves over $\Q$: 
\[
W_{s,s}:=X_0(Np)\times X_s\times X_s,
\]
where $X_0(Np)$ and $X_s$ are the classical modular curves attached to the congruence subgroups $\Gamma_0(Np)$ and $\Gamma_1(Np^s)$, respectively, and the model for the latter is the one for which the cusp $\infty$ is defined over $\bQ$. The group $G_s^{(N)}:=(\Z/Np^s\Z)^\times$ 
acts on $X_s$ by the diamond operators $\langle a;b\rangle$ ($a\in(\Z/N\Z)^\times$, $b\in(\Z/p^s\Z)^\times$), and we let 
\[
W_s:=W_{s,s}/D_s 
\]
be the quotient of $W_{s,s}$ by the action of the subgroup $D_s\subset G_{s}^{(N)}\times G_s^{(N)}$ consisting of elements of the form $(\langle a;b\rangle,\langle a;b^{-1}\rangle)$. Let $^{\flat}\Delta_{s,s}\in{\rm CH}^2(W_{s,s})(\bQ(\zeta_{s}))$ be the class in the Chow group defined by the ``twisted diagonal cycle'' defined in (41), and let $^{\flat}\Delta_{s}\in{\rm CH}^2(W_{s})(\bQ(\zeta_{s}))$ denote its natural image under the projection ${\rm pr}_s:W_{s,s}\rightarrow W_s$. By Proposition~1.4, after applying the correspondence $\varepsilon_{s,s}$ in (47) the cycle $\Delta_{s,s}$ becomes null-homologous, and so
\[
\Delta_s:=\varepsilon_{s,s}(^\flat\Delta_s)\in{\rm CH}^2(W_s)_0(\Q(\zeta_{s})),
\]
letting $\varepsilon_{s,s}$ still denote the linear endomorphism of ${\rm CH}^2(W_{s})$ defined by the above correspondence. Let $\varepsilon_s:G_\Q\rightarrow(\Z/p^s\Z)^\times$ be the mod $p^s$ cyclotomic character, and let $X_s^\dagger$ be the twist of $X_s$ by the cocycle $\sigma\in G_\Q\mapsto\langle 1;\epsilon_s(\sigma)\rangle$. By Proposition~1.6, we may alternatively view
\[
\Delta_s\in{\rm CH}^2(W_s^\dagger)_0(\bQ),
\]
where $W_{s}^\dagger$ the quotient of $W_{s,s}^\dagger:=X_0(Np)\times X_s\times X_s^\dagger$ be a diamond action defined as before. 

Consider the $p$-adic \'etale Abel--Jacobi  map
\[
{\rm AJ}_{\rm et}:{\rm CH}^2(W_s^\dagger)_0(\bQ)\to\rmH^1(\bQ,\rmH^3_{\rm et}({{W}_s^\dagger}_{/\Qbar},\Z_p)(2)).
\]
Let $e_{\rm ord}=\lim_nU_p^{n!}$ be Hida's ordinary projector. Set
\begin{equation}\label{eq:Vss}
V_{s,s}^{\rm ord}:=\rmH^1_{\rm et}({{X}_0(Np)}_{/\Qbar},\bZ_p)\ot e_{\rm ord}(\rmH^1({{X}_s}_{/\Qbar},\bZ_p)(1))\ot e_{\rm ord}(\rmH^1({{X}^\dagger_s}_{/\Qbar},\Z_p)(1)),
\end{equation}
and let $V^{\rm ord}_s:=(V_{s,s}^{\rm ord})_{D_s}$ denote the $D_s$-coinvariants. Let $\varpi_2:X_{s+1}\mapsto X_s$ be the degeneracy map given by $\tau\mapsto p\tau$ on the complex upper half plane, which naturally defines
\begin{equation}\label{eq:deg22}
(\varpi_{2,2})_*=(1,\varpi_2,\varpi_2)_*:V_{s+1,s+1}^{\rm ord}\to V_{s,s}^{\rm ord}.
\end{equation}
Let $\tilde{\kappa}_s\in\rmH^1(\Q,V_s^{\rm ord})$ denote the image of ${\rm AJ}_{\rm et}(\Delta_s)$ under the composite map
\begin{align*}
\rmH^1(\bQ,\rmH^3_{\rm et}({{W}_s^\dagger}_{/\Qbar},\Z_p)(2))&
\overset{\varepsilon_{s,s}{\rm pr}_{s,*}^{-1}\varepsilon_{s,s}}\rightarrow
\rmH^1(\bQ,\rmH^3_{\rm et}({{W}_s^\dagger}_{/\Qbar},\Z_p)_{D_s}(2))\\
&\overset{(1,e_{\rm ord},e_{\rm ord}){\rm pr}_{1,1,1}}\rightarrow\rmH^1(\Q,(V_{s,s}^{\rm ord})_{D_s})=\rmH^1(\Q,V_s^{\rm ord}),
\end{align*}
where the first arrow is defined by Lemma~1.8, and ${\rm pr}_{1,1,1}$ is the projection onto the $(1,1,1)$-component in the K\"unneth decomposition for $H^3_{\rm et}({{W}_s^\dagger}_{/\Qbar},\Z_p)$. By Proposition~1.9, we have $(\varpi_{2,2})_*(\tilde{\kappa}_{s+1})=(1,U_p,1)(\tilde{\kappa}_s)$, 
and hence we obtain the compatible family
\[
\kappa_\infty:=\varprojlim_s(1,U_p,1)^{-s}(\tilde{\kappa}_s)\in\rm H^1(\Q,\mathbb{V}_\infty^{\rm ord}),\quad\textrm{where}\;\mathbb{V}_\infty^{\rm ord}:=\varprojlim_{s}V_s^{\rm ord},
\]
with limit with respect to the maps induced by $(\ref{eq:deg22})$. The triple $(f,\bfg,\bfh)$ defines a natural projection $\varpi_{f,\bfg,\bfh}:\mathbb{V}_\infty^{\rm ord}\rightarrow\mathbb{V}_{f,\bfg\bfh}^\dagger(N)$, and following Definition~1.15 one sets
\[
\kappa(f,\bfg\bfh):=\varpi_{f,\bfg\bfh}(\kappa_\infty)\in \rmH^1(\Q,\mathbb{V}_{f,\bfg\bfh}^\dagger(N)); 
\]
this is the class in $(\ref{eq:1.15})$. Now, to prove the equality (\ref{eq:ERL}) in the theorem, it suffices to show that both sides agree at infinitely many points. Let $x\in\mathfrak{X}_\cR^+$ have weight $2$ with $\zeta:=\epsilon_x(1+p)\in\mu_{p^\infty}$ a primitive $p^s$-th root of unity, and set
\[
\kappa(f,\bfg_x\bfh_x):=\kappa(f,\bfg\bfh)\vert_{T=\zeta\mathbf{v}-1}.
\]  
Directly from the definitions (\emph{cf.} Proposition~2.5), we have
\begin{equation}\label{eq:control}
\kappa(f,\bfg_x\bfh_x)=a_p(\bfg_x)^{-s}\cdot\varpi_{f,\bfg_x,\bfh_x}({\rm AJ}_{\rm et}(\Delta_s))\in\rmH^1(\bQ,V_{f\bfg_x\bfh_x}(N)),
\end{equation}
where $V_{f\bfg_x\bfh_x}(N)$ is the $(f,\bfg_x,\bfh_x)$-isotypic component of $(\ref{eq:Vss})$, and $\varpi_{f,\bfg_x,\bfh_x}$ is the projection to that component. By Corollary~2.3 and (77), the image of $\kappa(f,\bfg_x\bfh_x)$ in the local cohomology group $\rmH^1(\bQ_p(\zeta),V_{f\bfg_x\bfh_x}(N))$ lands in the Bloch--Kato finite subspace $\rmH_{\rm fin}^1(\bQ_p(\zeta),V_{f\bfg_x\bfh_x}(N))\subset\rmH^1(\bQ,V_{f\bfg_x\bfh_x}(N))$, and so we may consider the image $\log_p(\kappa(f,\bfg_x\bfh_x))$ of this restriction under the Bloch--Kato logarithm map
\[
\log_p:\rmH_{\rm fin}^1(\bQ_p(\zeta),V_{f\bfg_x\bfh_x}(N))\to({\rm Fil}^0D_{f\bfg_x\bfh_x}(N))^\vee,
\]
where $D_{f\bfg_x\bfh_x}(N):=(B_{\rm cris}\otimes V_{f\bfg_x\bfh_x}(N))^{G_{\bQ_p(\zeta)}}$, and the dual is with respect to the de Rham pairing $\langle\;,\;\rangle_{\rm dR}$. By the de Rham comparison isomorphism, 
we have
\[
D_{f\bfg_x\bfh_x}(N)\iso\rmH^1_{\rm dR}(X_0(Np)_{/\Q_p(\zeta)})[f]\times
\rmH^1_{\rm dR}({X_s}_{/\Q_p(\zeta)})(1)[\bfg_x]\times\rmH^1_{\rm dR}({X_s}_{/\Q_p(\zeta)})(1)[\bfh_x].
\]
As in p.~639, attached to the test vectors $(\breve{f},\breve{\bfg}_x,\breve{\bfh}_x)$ one
has the de Rham classes $(\eta_{\breve{f}^*}^\circ, \omega_{\breve{\bfg}_x}^\circ,\omega_{\breve{\bfh}_x}^\circ)$, 
and comparing Proposition~2.10 and Corollary~2.11 we deduce from $(\ref{eq:control})$ that 
\begin{equation}\label{eq:4.14}
\begin{split}
\langle\log_p(\kappa(f,\bfg_x\bfh_x)),\eta_{\breve{f}^*}^\circ\otimes\omega_{\breve{\bfg}_x}^\circ\omega_{\breve{\bfh}_x}^\circ\rangle_{\rm dR}
&=a_p(\bfg_x)^{-s}\langle{\rm AJ}_p(\Delta_s),\eta_{\breve{f}^*}^\circ\otimes\omega_{\breve{\bfg}_x}^\circ\omega_{\breve{\bfh}_x}^\circ\rangle_{\rm dR}\\
&=\mathcal{E}(f,\bfg_x,\bfh_x)
\cdot\mathfrak{g}(\epsilon_x)\cdot\alpha_p^{s-1}a_p({\bfg_x})^{-s}a_p(\bfh_x)^{-s}\cdot\breve{f}^*(\breve{\bfg}_x\breve{H}^\iota_x),\nonumber
\end{split}
\end{equation}
where $\breve{H}_x^\iota=d^{-1}\breve{\bfh}_x^\iota$ is the primitive of $\breve{\bfh}_x^\iota$ given by part (3) of Corollary~4.5, and $\mathcal{E}(f,\bfg_x,\bfh_x)=-2(1-\alpha_p^{-1}a_p(\bfg_x)a_p(\bfh_x)^{-1})^{-1}$. Consider the formal $q$-expansion
\[
\breve{\boldsymbol{H}}^\iota(q):=\sum_{p\nmid n}\langle n^{-1}\rangle a_n(\breve{\bfh})q^n.
\]
Taking $(\breve{f},\breve{\bfg},\breve{\bfh})$ to be the test vectors $\breve{\boldsymbol{F}}^\star$ from Theorem~\ref{thm:triple-L} above, the construction in \cite[\S{3.6}]{hsieh-triple} yields $\mathscr{L}_p^f(\breve{f},\breve{\bfg}\breve{\bfh})=\breve{f}^*(\breve{\bfg}\breve{\boldsymbol{H}}^\iota)$. Since by construction $\breve{\boldsymbol{H}}^\iota$ specializes at $x$ to $\breve{H}_x^\iota$, we thus see as in the proof of Theorem~4.16 that
\begin{equation}\label{eq:an}
\langle\log_p(\kappa(f,\bfg_x\bfh_x)),\eta_{\breve{f}^*}^\circ\otimes\omega_{\breve{\bfg}_x}^\circ\omega_{\breve{\bfh}_x}^\circ\rangle_{\rm dR}
=\mathcal{E}(f,\bfg_x,\bfh_x)
\cdot\mathfrak{g}(\epsilon_x)\cdot\alpha_p^{s-1}a_p(\bfg_x)^{-s}a_p(\bfh_x)^{-s}\cdot\mathscr{L}_p^f(\breve{f},\breve{\bfg}\breve{\bfh})(x).
\end{equation}

On the other hand, letting $\psi_x:=\Psi_T|_{T=\zeta\mathbf v-1}$, we obtain that $(\bfg_x,\bfh_x)$ is a pair of theta series attached to the characters $(\chi\psi^{-1}_x,\chi^{-1}\psi_x^{-1})$ of $G_K$ with $a_p(\bfg_x)=\chi\psi^{-1}_x(\sigma_{\ol{\frakp}})$ and $a_p(\bfh_x)=\chi^{-1}\psi^{-1}_x(\sigma_{\ol{\frakp}})$.  Moreover, we have 
\[\ep_x|_{G_{\Qp}}=\psi_x^{1+c}|_{G_{K_\frakp}}\cdot \varepsilon_{\rm cyc}^{-1};\quad 
\psi_x^{c-1}=\phi_x\breve\varepsilon_\cF^{-1}
\] 
for some finite order character $\phi_x$ of ${\rm Gal}(F_\infty/\Qp)$, viewing the character in the left-hand side of this equality as character on ${\rm Gal}(F_\infty/F)$ by composition with ${\rm Gal}(F_\infty/F)
\subset{\rm Gal}(F_\infty/\Q_p)\twoheadrightarrow{\rm Gal}(K_{\infty,\pp}/K_\pp)\subset\Gamma_\infty$. 
Setting $\eta=\eta_{\breve{f}^*}^\circ\ot t^{-1}$ and  $\mathbf{z}_x=\loc_\pp(\kappa(\breve{f}^\star,\breve{\bfg}^\star\breve{\bfh}^\star))_x$, we thus see that 
\begin{equation}\label{eq:al}
\begin{split}
\langle\log_p(\kappa(f,\bfg_x\bfh_x)),\eta_{\breve{f}^*}^\circ\otimes\omega_{\breve{\bfg}_x}^\circ\omega_{\breve{\bfh}_x}^\circ\rangle_{\rm dR}
&=\langle\log_p(\mathbf{z}_{x})\otimes t,\eta\rangle_{\rm dR}\\
&=\mathfrak{g}(\epsilon_x)\cdot\alpha_p^sa_p(\bfg_x)^{-s}a_p(\bfh_x)^{-s}\cdot{\rm Col}^\eta(\mathbf{z}_x)(\psi_x^{c-1}),
\end{split}
\end{equation}
using Theorem~\ref{thm:kobayashi} with $j=-1$ for the last equality. Comparing $(\ref{eq:an})$ with $(\ref{eq:al})$ and letting $s$ vary, the result follows.
\end{proof}



We can now immediately deduce the following key cohomological construction of $\Theta_{f/K}$:

\begin{thm}\label{thm:theta-Euler}
	With notations and assumptions as in Theorem~\ref{thm:DR-erl}, we have
	\[
\COL^\eta({\rm loc}_\pp(\kappa(\breve{f}^\star,\breve{\bfg}^\star\breve{\bfh}^\star)))=\pm \Psi_T^{c-1}(\sigma_{\mathfrak{N}^+})\cdot\Theta_{f/K}(T)\cdot \sqrt{L^{\rm alg}(E/K\otimes\chi^2,1)}\cdot \frac{2C_{f,\chi}}{\al_p(1-\al_p^{-1}\chi(\ol{\frakp})^2)},
	\]
	where $C_{f,\chi}\in K(\chi,\al_p)^\x$ is the non-zero algebraic number as in Proposition~\ref{prop:factor}. 
\end{thm}
\begin{proof}Note that $a_p(\bfg)a_p(\bfh)^{-1}=\chi(\ol{\frakp})^2$. The theorem thus follows immediately from Proposition~\ref{prop:factor} and Theorem~\ref{thm:DR-erl}. 
\end{proof}

\subsection{Generalized Kato classes}

Set $\alpha=\chi(\overline{\pp})$, 
and denote by $(g_\alpha, h_{\alpha^{-1}})$ the
weight $1$ forms obtained by specializing the Hida families $(\bfg, \bfh)$ at $S=\mathbf{v}-1$. Thus $g_\alpha$ (resp. $h_{\alpha^{-1}}$) is the $p$-stabilization of the theta series $g=\theta_\chi$ (resp. $h=\theta_{\chi^{-1}}$) having   $U_p$-eigenvalue $\alpha$ (resp. $\alpha^{-1}$). 
By specialization, the $\cO\dBr{S}$-adic class in $(\ref{eq:prF})$ yields the class
\begin{equation}\label{eq:aa}
\kappa(f,g_\alpha,h_{\alpha^{-1}}):=\kappa(\breve{f},{\breve{\bfg}\breve{\bfh}})\vert_{S=\mathbf{v}-1}\in \rmH^1(\bQ,V_{fgh}),\nonumber
\end{equation}
where $V_{fgh}:=V_f\otimes V_g\otimes V_h$. Setting $\beta=\chi(\pp)$ and alternatively changing the roles of $\pp$ and $\overline{\pp}$ in the construction $\bfg$ and $\bfh$ we thus obtain the four 
\emph{generalized Kato classes} 
\begin{equation}\label{eq:gen-Kato}
\kappa(f,g_\alpha,h_{\alpha^{-1}}),\;\kappa(f,g_\alpha,h_{\beta^{-1}}),\;
\kappa(f,g_\beta,h_{\alpha^{-1}}),\;\kappa(f,g_\beta,h_{\beta^{-1}})\in\rmH^1(\bQ,V_{fgh}).
\end{equation}

From now on, we assume that $\alpha\neq\pm{1}$, so that the four classes $(\ref{eq:gen-Kato})$ are \emph{a priori} distinct. 
%
Recall that $f$ is the $p$-stabilization of the newform associated to an elliptic curve $E/\bQ$, 
so that 
$V_f(1)\simeq V_pE$, and 
let $\kappa_{\alpha,\alpha^{-1}}, \kappa_{\alpha,\beta^{-1}}, \kappa_{\beta,\alpha^{-1}}, \kappa_{\beta,\beta^{-1}}\in
\rmH^1(K,V_pE\otimes L)$ be the image of the classes $(\ref{eq:gen-Kato})$ under the map $\rmH^1(\bQ,V_{fgh})\rightarrow\rmH^1(K,V_pE\otimes L)$ induced by $(\ref{eq:triple})$. 

\begin{cor}\label{cor:DR-erl}
Assume that $L(E/K,1)=0$ and that $L(f/K\ot\chi^2,1)\neq 0$. Then:
\begin{enumerate} 
\item $\kappa_{\alpha,\alpha^{-1}}, \kappa_{\beta,\beta^{-1}}\in\Sel(K,V_pE\ot L)$. 
\item $\kappa_{\alpha,\beta^{-1}}=\kappa_{\beta,\alpha^{-1}}=0$.
\end{enumerate}
\end{cor}

\begin{proof}
By the factorization $(\ref{eq:factorL})$, the inclusions in part $(1)$ follow from the proof of Lemma~\ref{lem:ERL}. To see part $(2)$, we make use of the $3$-variable generalized Kato class 
\[
\bfkappa:=\bfkappa(\bff,\bfg,\bfh')(S_1,S_2,S_2) \in\rmH^1(\bQ,\mathbb{V}_{\bff\bfg\bfh'}^\dagger)
\]
defined in \cite[\S 3.7 (119)]{DR3} attached to the triple $\bff=\bff(S_1)$, $\bfg=\boldsymbol{\theta}_{\chi}(S_2)$ and $\bfh'=\boldsymbol{\theta}_{\chi}(S_3)$. Thus $\kappa(f,g_\alpha,h_{\beta^{-1}})$ is the specialization $\bfkappa((1+p)^2-1,\mathbf{v}-1,\mathbf{v}-1)$. Let \[\bfkappa':=\bfkappa((1+p)^2-1,\mathbf v(1+T)-1,\mathbf v(1+T)^{-1}-1)\in\rmH^1(\Q,\mathbb V_{f\bfg\bfh'}^\dagger),\] 
where $\mathbb V_{f\bfg\bfh'}^\dagger\iso V_pE\ot\bigl(\Ind_K^\Q\chi^2\oplus \Ind_K^\Q\Psi_T^{1-c}\bigr)$. 
As in Lemma~\ref{lem:ERL}, by \cite[Prop.~3.28]{DR3} the class $\loc_p(\bfkappa')$ belongs to $\rmH^1(\Q_p,F^+\mathbb{V}_{f\bfg\bfh'}^\dagger)$, where
\[ 
F^+\mathbb{V}_{f\bfg\bfh'}^\dagger=V_pE\ot\chi^{-2} + F^0V_pE\ot(\Psi^{1-c}_{T}\oplus \Psi^{1-c}_{T}).
\]
It follows that the projection $\bfkappa'_V$ of $\bfkappa'$ into $\wh{\rmH}^1(K_\infty,V_pE)$ is crystalline at $p$, and hence $\bfkappa'_V$ is a Selmer class for $V_pE$ over the anticyclotomic $\bZ_p$-extension $K_\infty/K$. Since the space of such universal norms is trivial by Cornut--Vatsal \cite{CV-doc} (the sign of $E/K$ is $+1$ in our case), this shows that  $\bfkappa'_V=0$ and therefore $\kappa(f,g_\alpha,h_{\beta^{-1}})=\kappa_{\alpha,\beta^{-1}}=0$. The vanishing of $\kappa_{\beta,\alpha^{-1}}$ is shown in the same manner.
\end{proof}

\section{Proof of the main theorem}

\begin{proof}[Proof of Theorem~A]
 Let $V=V_pE\ot_{\Qp}L$ and $S={\rm Sel}(K,V)$. Let $S^{(r)}:=S_p^{(r)}(E/K)\otimes_{\bQ_p}L$ be the subspaces in \S\ref{SS:derived} and $S^{(\infty)}$ be the subspace of anticyclotomic universal norms. By \cite[Thm.~4.2]{howard-derived} we have  the filtration \begin{equation}\label{eq:howard-fil}
S=S^{(1)}\supset S^{(2)}\supset\cdots\supset S^{(r)}\supset S^{(r+1)}\supset\cdots\supset S^{(\infty)},
\end{equation}
and $S^{(r+1)}$ is the null space of the $r$-th derived height pairing $h^{(r)}\colon S^{(r)}\times S^{(r)}\to J^r/J^{r+1}\otimes L$. 
In addition, by \cite[Remark 1.12, Thm.~4.2]{howard-derived} we have for every $x,y\in S^{(r)}$
\beq\label{E:alt}
h^{(r)}(x,y)=(-1)^{r+1}h^{(r)}(y,x);\quad 
h^{(r)}(x^\tau,y^\tau)=(-1)^rh^{(r)}(x,y), 
\eeq
where $\tau$ denotes the complex conjugation. 
Since $L(E^K,1)\neq 0$, we have ${\rm Sel}(\bQ,V_pE^K)=\{0\}$ by Kolyvagin's work \cite{kol88} (or Kato's \cite{Kato295}), and so letting $S^+$
be the subspace of $S$ fixed by $\tau$, this shows 
\[
S=S^+.
\] 
By \eqref{E:alt}, this implies that $h^{(1)}$ is identically zero, and so $S=S^{(2)}$; by the same argument, $S^{(r)}=S^{(r+1)}$ for every odd $r\geqslant 1$. On the other hand, $S^{(\infty)}=\{0\}$ by Cornut--Vatsal \cite{CV-doc}.
 
We first prove the implication ${\rm (ii)}\Rightarrow{\rm (i)}$. Suppose that $\dim_{\Qp}{\rm Sel}_{\rm str}(\bQ,V_pE)=1$. By $p$-parity \cite{nekovarII}, this implies that $\dim_{\Qp}{\rm Sel}(\bQ,V_pE)=2$ and the composite map 
\[
\log_{\omega_E,\pp}:S={\rm Sel}(K,V)\overset{\loc_\pp}\rightarrow\rmH^1_{\rm fin}(K_\pp,V)\overset{(\ref{log})}\rightarrow L
\] 
is nonzero. Under our hypotheses we have
\[ 
{\rm dim}_{L}S={\rm dim}_{\Q_p}{\rm Sel}(\bQ,V_pE)=2.
\] 
In view of \eqref{E:alt}, we thus find that $(\ref{eq:howard-fil})$ reduces to
\begin{equation}\label{eq:fil-Howard-red}
S=S^{(1)}=S^{(2)}=\cdots= S^{(r)}\quad\textrm{and}\quad S^{(r+1)}=\cdots= S^{(\infty)}=\{0\}
\end{equation}	
for some (even) integer $r\geqslant 2$ and the derived $p$-adic height $h^{(r)}$ is a non-degenerate pairing on $S^{(r)}$. 

Let $X_\infty$ 
be the Pontryagin dual of ${\rm Sel}_{p^\infty}(E/K_\infty)$, which is known to be  $\Lambda$-torsion \cite{bdIMC}. Let $J\subset\Lambda$ be the augmentation ideal, and fix a pseudo-isomorphism
\begin{equation}\label{eq:ps}
X_\infty\sim M\oplus M',\quad\textrm{with}\quad M\iso(\Lambda/J)^{e_1}\oplus(\Lambda/J^2)^{e_2}\oplus\cdots
\end{equation}
with $M'$ a torsion $\Lambda$-module having characteristic ideal prime to $J$. By \cite[Cor.~4.3(c)]{howard-derived} we have $e_i={\rm dim}_{L}(S^{(i)}/S^{(i+1)})$; letting  
$\mathcal{L}_p\in\Lambda$ be a generator of the principal ideal 
${\rm char}_\Lambda(X_\infty)$, combining $(\ref{eq:fil-Howard-red})$ and $(\ref{eq:ps})$ this shows that
\begin{equation}\label{eq:r}
{\rm ord}_{J}\mathcal{L}_p=2r.\nonumber
\end{equation}
On the other hand, by our hypotheses on $\bar\rho_{E,p}$ the divisibility in the Iwasawa main conjecture due to Skinner--Urban \cite{SU} (see [\emph{loc.cit.}, \S{3.6.3}]) implies that $(\Theta_{f/K}^2)\supset(\mathcal{L}_p)$, and so 
\begin{equation}\label{eq:ineq}
r\geqslant\rho:={\rm ord}_{J}(\Theta^{}_{f/K}).
\end{equation}
Let $\bar\theta_{f/K}$ be the \emph{leading coefficient} of $\Theta_{f/K}$ defined by \[
\bar\theta_{f/K}:=\Theta^{}_{f/K}(T)\pmod{J^{\rho+1}}\in J^\rho/J^{\rho+1}.
\]
From (\ref{eq:fil-Howard-red}) and (\ref{eq:r}) we see that $S=S^{(\rho)}$. Thus combining the derived $p$-adic height formula in Proposition~\ref{P:37.H}, Theorem~\ref{thm:theta-Euler}, and part (2) of Lemma~\ref{lem:ERL}, we deduce that for every $x\in S^{(\rho)}=S$ we have
\beq\label{E:1.Main}
h^{(\rho)}(\kappa_{\alpha,\alpha^{-1}},x)=\frac{1-p^{-1}\alpha_p}{1-\alpha_p^{-1}}\cdot\bar\theta_{f/K}\cdot\log_{\omega_E,\pp}(x)\cdot  C,
\eeq
where $\alpha_p$ is the $p$-adic unit root of $X^2-a_p(E)X+p=0$ and $C$ is a non-zero algebraic number with $C^2\in K(\chi,\al_p)^\x$. 
Since $\bar\theta_{f/K}\neq 0$ and as noted above our hypotheses imply that the map $\log_{\omega_E,\pp}$ is nonzero, we see that $r=\rho$ and the non-vanishing of $\kappa_{\alpha,\alpha^{-1}}$ follows. 

Now we proceed to establish the implication ${\rm (i)}\Rightarrow{\rm (ii)}$. Suppose that $\kappa_{\al,\al^{-1}}\neq 0$.  We shall prove that $\dim_L S=2$ and $\log_{\omega_E,\pp}$ is a nonzero map. Consider again the filtration (\ref{eq:howard-fil}).
Then the combination of Proposition~\ref{P:37.H}, Theorem~\ref{thm:theta-Euler}, and part (2) of Lemma~\ref{lem:ERL} shows that the class $\kappa_{\alpha,\alpha^{-1}}$ belongs to $S^{(\rho)}$ with $\rho={\rm ord}_J(\Theta_{f/K})$; in particular, $S^{(\rho)}\neq 0$. With notations as in (\ref{eq:ps}), this implies that $e_{r_0}\neq 0$ for some \emph{even} $r_0\geqslant\rho$, and so $e_{r_0}\geqslant 2$ by (\ref{E:alt}). 
On the other hand, we have \[2\rho\geqslant e_1+2e_2+\cdots+re_r+\cdots\] according to the divisibility in the Iwasawa main conjecture due to Bertolini--Darmon \cite{bdIMC} (see also \cite{pollack-weston}). This implies that $e_{\rho}=2$ and $e_r=0$ if $r\neq \rho$. We thus conclude that ${\rm dim}_{L}S=2$ and $h^{(\rho)}$ is non-degenerate on $S^{(\rho)}=S$. In light of $(\ref{E:1.Main})$, this shown that the map ${\rm log}_{\omega_E,\pp}$ is non-zero, yielding the proof of the implication ${\rm (i)}\Rightarrow{\rm (ii)}$. 
\end{proof}

The following is an immediate consequence of the height formula \eqref{E:1.Main}: \begin{cor}\label{C:51.M}The class $\kappa_{\alpha,\alpha^{-1}}\mod{\Qbar^\x}$ depends only on $K$, not on the auxiliary choice of ring class character $\chi$. Moreover, as elements in $E(\Q)\ot_{\Z}L$, we have
\[\kappa_{\alpha,\alpha^{-1}}=C\cdot \frac{1-p^{-1}\alpha_p}{1-\alpha_p^{-1}}\cdot\frac{\bar\theta_{f/K}}{h^{(\rho)}(P,Q)}\cdot\left(P\ot \log_pQ-Q\ot \log_pP\right)\]
for any basis $(P,Q)$ of $E(\Q)\ot_\Z \Q$.
\end{cor}


This suggests that the refined conjecture \cite[Conj.~3.12]{DR2.5} in this case should follow from the $p$-adic Birch--Swinnerton-Dyer formula of \cite[Conj.~4.3]{BDmumford-tate}.

\section{Numerical examples}\label{S:example}
In this section, we give examples of elliptic curves of rank $2$ having non-trivial generalized Kato classes. To be more precise, we consider elliptic curves $E/\Q$ with 
\[
\Ord_{s=1}L(E,s)=\rank_{\Z}E(\Q)=2
\] 
and conductor $N\in\{q,2q\}$ with $q$ an odd prime. We take a square-free integer $-\Delta<0$ such that $K=\Q(\sqrt{-\Delta})$ has class number one, $q$ is inert in $K$, and $L(E^K,1)\neq 0$, and take a prime $p>3$ of good ordinary prime for $E$ which is split in $K$ and such that $E[p]$ is an  irreducible $G_{\bQ}$-module. 
For every triple $(E,p,-\Delta)$, letting $f\in S_2(\Gamma_0(N))$ be the newform associated to $E$, we give numerical examples where the associated theta element 
\[
\Theta_{E/K}(T)=\Theta_{f/K}(T)\in\Z_p\dBr{T}
\] 
vanishes to order exactly $2$ at $T=0$. When that is the case, by the work of Bertolini--Darmon \cite{bd-117,bdIMC} 
on the anticyclotomic Iwasawa main conjecture\footnote{As extended by Pollack--Weston \cite{pollack-weston} to allow for weaker hypotheses.} (see  \cite[Cor.~3]{bdIMC}), 
it follows that $\Sha(E/K)[p^\infty]$ is finite. Moreover, 
the residual representation 
$E[p]$ must ramify at $N^-=q$ by \cite[Thm.~1.1]{ribet-eps} and we checked that $E[p]$ is irreducible, either by \cite{mazur-prime} for $p\geqslant 11$ or checking that the elliptic curves we consider have no rational $m$-isogeny for $m>3$ according to Cremona's tables. Thus for every ring class character $\chi$ with $L(E/K,\chi^2,1)\neq 0$ (as one can always find by virtue of \cite[Thm.~1.4]{vatsal-special}, as extended in \cite[Thm.~D]{ChHs1}), the examples below provide instances where the generalized Kato class $\kappa_{E,K}$ is a nonzero class in the $2$-dimensional ${\rm Sel}(\bQ,V_pE)$ by virtue of Corollary~B. 

To explain these numerical examples, we prepare some notation. Let $B/\mathbf{Q}$ be the definite quaternion algebra of discriminant $q$. Let $R$ be an Eichler order of level $N/q$ and ${\rm Cl}(R)$ be the class group of $R$. Let $f_E:{\rm Cl}(R)\to\Z$ be the ($p$-adically normalized) Hecke eigenfunction associated with $f$ by the Jacquet--Langlands correspondence. Fix an optimal embedding $\cO_K\hookto R$ and an isomorphism $i_p:R\ot\Zp\iso {\rm M}_2(\Zp)$ such that $i_p(K)$ lies in the subspace of diagonal matrices. For $a\in\Zp^\x$ and an integer $n$, put \[r_n(a)=i_p^{-1}(\begin{pmatrix} 1&ap^{-n}\\ 0&1\end{pmatrix})\in \wh B^\x,\quad \wh B:=B\ot_\Z\wh\Z.\]Consider the sequence $\stt{P^a_n}_{n=0,1,\ldots}$ of right $R$-ideals defined by $P_n^a:=(r_n(a)\wh R)\cap B$. The images of these ideals $P_n^a$ in $\Cl(R)$ are usually referred to Gross points of level $p^n$. Letting $\bfu=1+p$, we
define the $n$-th theta element $\Theta_{E/K,n}(T)\in\Z_p[T]$ by 
\[\Theta_{E/K,n}(T):=\frac{1}{\al_p^{n+1}}\sum_{i=0}^{p^n-1}\sum_{a\in \mu_{p-1}}\left(\al_p\cdot f_E(P_{n}^{a\bfu^i})-f_E(P_{n+1}^{a\bfu^i})\right)(1+T)^{i}.\]
By the definition of theta elements in \cite[\S 2.7]{BDmumford-tate}, if $K$ has class number one, we then have 
\begin{align*}
\Theta_{E/K}(T)&\con\Theta_{E/K,n}(T)\pmod{(1+T)^{p^n}-1}.
\end{align*}
Since $(p^n,(1+T)^{p^n}-1)\subset (p^n,T^p)$ and $p>2$, to check the vanishing $\Theta_{E/K}(T)$ to exact order $2$ at $T=0$, it suffices to compute $\Theta_{E/K,n}(T)$ for sufficiently large $n$. The following examples were obtained by implementing the Brandt module package in SAGE.
\begin{table}[htbp!]
		\centering
\begin{tabu}
{ | X[3,c] | X[1,c] | X[2,c] | X[30,l] | }
\hline
$E$ & $p$&$-\Delta$ &$\Theta_{E/K,2}(T)\;{\rm mod}\;{(p^2,T^p)}$\\
\hline\hline
389a1&$11$&$-2$&$10T^2+69T^3+T^4+103T^5 +106T^6+66T^7+ 11T^8+ 55T^9+110T^{10}$ 
\\
433a1&$11$&$-7$&$88T^2+22T^3+86T^4+7T^5+10T^6+12T^7+29T^8+88T^9+48T^{10}$\\
446c1&7&$-3$&$22T^2 + 27T^3 + 3T^4 + 16T^5 + 11T^6$\\
563a1&$5$&$-1$&$18T^2+9T^3+5T^4$\\
643a1&$5$&$-1$&$T^2+21T^4$\\
709a1&11&$-2$&$27T^2 + 114T^3 + 3T^4 + 14T^5 + 36T^6 + 15T^7 + 42T^8 + 44T^9 + 91T^{10}$\\
718b1&5&$-19$&$3T^2 + 20T^3 + 12T^4$\\
794a1&7&$-3$&$47T^2 + 23T^3 + 8T^4 + 24T^5 + 7T^6$\\
997b1&$11$&$-2$&$71T^2 + 41T^3 + 83T^4 + 19T^5 + 114T^6 + 111T^7 + 101T^8 + 46T^9 + 102T^{10}$\\
997c1&$11$&$-2$&$54T^2 + 38T^3 + 36T^4 + 81T^5 + 82T^6 + 18T^7 + 72T^8 + 95T^9 + 4T^{10}$\\
1034a1&$5$&$-19$&$22T^2 + 4T^3 + 6T^4$\\
1171a1&$5$&$-1$&$6T^2+6T^3+20T^4$\\
1483a1&$13$&$-1$&$128T^2+148T^3+ 127T^4 +162T^5+30T^6+149T^7+ 141T^8+97T^9+49T^{10}+13T^{11} +29T^{12}$\\
1531a1&$5$&$-1$&$16T^2+7T^3+21T^4$\\
1613a1&$17$&$-2$&$128T^2 + 165T^3 + 224T^4 + 287T^5 + 140T^6 + 211T^7 + 147T^8 + 160T^9 + 59T^{10} + 122T^{11} + 195T^{12} + 43T^{13} + 207T^{14} + 214T^{15} + 285T^{16} $\\
1627a1&$13$&$-1$&$101T^2+151T^3+58T^4+104T^5+3T^6+165T^7+128T^8+63T^9+17T^{10}+55T^{11}+166T^{12}$\\
1907a1&$13$&$-1$&$72T^2+131T^3+32T^4+142T^5+84T^6+104T^7+90T^8+105T^9+38T^{10}+92T^{11}+116T^{12}$\\
1913a1&$7$&$-3$&$41T^2 + 16T^3 + 28T^4 + 23T^5 + 14T^6 $\\
2027a1&$13$&$-1$&$54T^2 + 128T^3 + 65T^4 + 93T^5 + 83T^6 + 161T^7 + 113T^8 + 133T^9 + 49T^{10} + 151T^{11} + 13T^{12}$ \\
\hline
\end{tabu}
\end{table}

\begin{table}[htbp!]
		\centering
		\begin{tabu} { | X[3,c] | X[1,c] | X[2,c] | X[30,l] | }
			\hline
			$E$ &$p$ & $-\Delta$& $\Theta_{E/K,3}(T)\;{\rm mod}\;{(p^3,T^p)}$\\
			\hline\hline
			571b1&5&$-1$&$100T^2+100T^3+15T^4$\\
			1621a1&11&$-2$&$1089T^2+807T^4+ 986T^5+ 586T^6+ 1098T^7+ 772T^8+ 228T^9 +1296T^{10}$\\
			\hline
		\end{tabu}
\end{table}



\bibliographystyle{amsalpha}
\bibliography{Kato-refs}

\end{document}